\documentclass[letterpaper,11pt]{article}%
\usepackage[margin=1in]{geometry}
\usepackage{color, xcolor, graphicx}
\usepackage[bookmarks, colorlinks=true, plainpages = false, citecolor=
darkblue, linkcolor = chicago-maroon, anchorcolor = red, urlcolor =
darkblue]{hyperref}
\usepackage{url}
\usepackage{amsmath, amsfonts, amsthm, amssymb, bm, bbm, verbatim, dsfont,
mathtools, mathrsfs}
\usepackage{orcidlink}
\usepackage{etoolbox}
\usepackage{titlesec}
\usepackage{almendra}
\usepackage{authblk}
\usepackage{array}
\usepackage{multirow}
\usepackage{multicol}
\usepackage{float}
\usepackage{pgf,tikz}
\usepackage{tikz-cd}
\usepackage{subcaption}
\usepackage[nameinlink]{cleveref}
\usepackage[round]{natbib}
\usepackage[utf8]{inputenc}
\usepackage[T1]{fontenc}
\usepackage{nicefrac}
\usepackage{pdfpages}
\usepackage{lmodern}
\usepackage{tabu}
\usepackage{enumitem}
\usepackage{setspace}
\usepackage{bbold}
\usepackage{amsmath}
\usepackage{amsfonts}
\usepackage{amssymb}
\usepackage{graphicx}%
\providecommand{\U}[1]{\protect\rule{.1in}{.1in}}
\definecolor{chicago-maroon}{RGB}{128,0,0}
\definecolor{darkblue}{rgb}{0.0, 0.0, 0.55}
\providecommand{\keywords}[1]{\textbf{Keywords: } #1}
\usetikzlibrary{arrows,shapes.arrows,shapes.geometric,shapes.multipart,fit,automata,decorations.pathmorphing,positioning}
\tikzset{
		>=stealth',
		true/.style={
		rectangle,
		draw=black, very thick,
		text width=6.5em,
		minimum height=2em,
		text centered,
		fill=gray, opacity = 0.5},
	punkt/.style={
		rectangle,
		rounded corners,
		draw=black, very thick,
		text width=6.5em,
		minimum height=2em,
		text centered},
	est/.style={
		circle,
		draw=black, very thick,
		text centered},
	shade/.style={
		circle,
		draw=black, very thick, fill=gray!50,
		text centered},
	weight/.style={
		circle,
		draw=black, very thick,
		text width=6.5em,
		minimum height=2em,
		text centered},
		pil/.style={
		->,
		thick,
		shorten <=2pt,
		shorten >=2pt,},
	double/.style={
		<->,
		thick,
		shorten <=2pt,
		shorten >=2pt,},
	dash/.style={
		dashed,
		thick,
		shorten <=2pt,
		shorten >=2pt,},
	dashdouble/.style={
		<->,
		dashed,
		thick,
		shorten <=2pt,
		shorten >=2pt,}
}
\makeatletter 
\newcolumntype{C}[1]{>{\centering\arraybackslash}p{#1}}

\def\E{\mathsf{E}}
\def\var{\mathsf{var}}
\def\cov{\mathsf{cov}}
\def\mse{\mathsf{mse}}

\renewcommand{\hat}{\widehat}

\newcommand{\myreferences}{Master.bib}
\theoremstyle{plain}
\newtheorem{theorem}{Theorem}
\newtheorem{lemma}{Lemma}

\newtheorem{corollary}{Corollary}

\theoremstyle{definition}
\newtheorem{definition}{Definition}

\newtheorem{assumption}{Assumption}

\newtheorem{remark}{Remark}

\newtheorem*{remark*}{Remark}
\def\bbP{\mathbb{P}}

\def\sfE{\mathsf{E}}

\def\bbR{\mathbb{R}}
\def\bbU{\mathbb{U}}
\def\calP{\mathcal{P}}
\def\calH{\mathcal{H}}
\def\bias{\mathsf{bias}}

\def\independenT#1#2{\mathrel{\rlap{$#1#2$}\mkern2mu{#1#2}}}
\def\op{\mathrm{op}}
\def\I{\mathrm{I}}
\def\Tr{\mathsf{Tr}}
\def\diff{\mathrm{d}}

\def\eps{\epsilon}

\def\mse{\mathsf{mse}}
\def\sfPi{\mathsf{\Pi}}
\def\SA{\mathrm{SA}}
\def\AL{\mathrm{AL}}

\begin{document}

\title{On the Asymptotic Inadmissibility of Double Machine Learning Estimators Under Structure-Agnostic Models}

\author[1]{Lin Liu\thanks{Email: \href{linliu@sjtu.edu.cn}{linliu@sjtu.edu.cn}.}\orcidlink{0000-0002-9883-7962}}

\author[2]{Rajarshi Mukherjee\thanks{Email: \href{ram521@mail.harvard.edu}{ram521@mail.harvard.edu}.}\orcidlink{0000-0002-5761-8958}}

\author[3]{James M. Robins\thanks{Email: \href{robins@hsph.harvard.edu}{robins@hsph.harvard.edu}. R Mukherjee's research is supported by NSF CAREER Award 8529216-01. The authors are grateful for the programme \emph{Causal inference: From theory to practice and back again} held in the Isaac Newton Institute of Mathematical Sciences, Cambridge, during which this work is completed. The current version of the paper strengthens the pointwise asymptotic (in)admissibility results of the original arXiv version to  uniform.}\orcidlink{0000-0001-6609-209X}}

\affil[1]{Institute of Natural Sciences, MOE-LSC, School of Mathematical Sciences, SJTU-Yale Joint Center for Biostatistics and Data Science, Shanghai Jiao Tong University}

\affil[2]{Department of Biostatistics, Harvard T. H. Chan School of Public Health}

\affil[3]{Department of Epidemiology and Department of Biostatistics, Harvard T. H. Chan School of Public Health}

\date{}
\maketitle

\vspace{-2em}

\begin{abstract}
Structure-agnostic (SA) models introduced by \citet{balakrishnan2026fundamental} aim to reflect the general lack of knowledge of structural assumptions on data-generating laws such as smoothness or sparsity in practice. Roughly speaking, SA models restrict the observed-data generating law to be in some $r_{n}$-neighborhood of (black-box machine learning) estimates, treated as given and fixed, where $r_{n}$ encodes the convergence rates of the estimates to the truth. Under SA models, \citet{balakrishnan2026fundamental} show that the popular Double Machine Learning (DML) estimators for three functionals, the quadratic functional in the Gaussian sequence model, the quadratic density integral functional and the expected conditional covariance, are minimax. However, minimax estimators may be inadmissible. In this paper, we show that, for the first two of the three functionals, the DML estimator is (uniformly) asymptotically inadmissible under the SA model. In particular, we show that these two functionals fall into a class of functionals, which we refer to as the \emph{monotone bias class}. For this class, we exhibit second-order ($U$-statistic) estimators, which asymptotically dominate DML estimators under the SA model. These second-order estimators are empirical higher-order influence function (HOIF) estimators introduced in \citet{liu2017semiparametric}. Furthermore, both the DML estimator and the second-order estimator are minimax for the expected conditional covariance, though neither asymptotically dominates the other. We also compare the SA model with the assumption-lean (AL) model of \citet{liu2020nearly, liu2024assumption} that imposes no assumptions beyond the trivial and empirically untestable hypothesis that the bias of any estimator may be of order $1$. As a consequence, under the AL model, a Wald confidence interval centered at a DML estimator may under-cover. \citet{liu2024assumption} introduced a class of valid tests that can falsify, for functionals in the \emph{monotone bias class}, the hypothesis that a DML-estimator-centered Wald interval covers the truth at its nominal level or greater. However, our tests are not consistent under the AL model, because no consistent tests exist \citep{robins1997toward}. Finally, for any functional with the mixed bias property of \citet{rotnitzky2021characterization}, such as the expected conditional covariance or the average treatment effect \citep{jin2025structure}, the above tests can falsify the hypothesis of \emph{rate-double-robustness}.
\end{abstract}


\onehalfspacing
\allowdisplaybreaks
\nopagebreak

\section{Introduction}
\label{sec:intro}

In scientific disciplines such as epidemiology, clinical medicine and economics, one of the most important statistical tasks is to infer from the observed data a low-dimensional, smooth functional $\psi (\theta)$ of the underlying data-generating law $\bbP_{\theta}$ posited to belong to a statistical model denoted by
\begin{align*}
\calP \equiv \calP (\Theta) \coloneqq \left\{ \bbP_{\theta}: \theta \in \Theta \right\}.
\end{align*}
Here we parameterize the data generating law by $\theta \in \Theta$. Without an essential loss of generality, we take $\psi \equiv \psi (\theta) \in \bbR$. Throughout this paper, we let $n$ denote the sample size and use $\psi$ and $\theta$ to denote the true values, which should cause no confusion.

Many examples of $\psi$, such as the average treatment effect under ignorability, are of substantive interest in practice. To avoid model misspecification bias, it is natural to take $\Theta$ to be high- or even infinite-dimensional and estimate $\theta$ nonparametrically by kernels or series in classical statistics. In terms of $\psi$, it has become common practice to construct the so-called Double Machine Learning (DML) estimator $\hat{\psi}_{1, n} \equiv \hat{\psi}_{1, n} (\hat{\theta})$ based on the first-order influence function of $\psi$ \citep{newey1990semiparametric, scharfstein1999adjusting, ai2003efficient, chernozhukov2018double, shi2026and}. Owing to the curse-of-dimensionality, uniformly consistent estimators exist neither for $\theta$ nor for $\psi$ without any additional assumptions on $\Theta$ \citep{stone1980optimal, stone1982optimal, ritov1990achieving, robins1997toward}. For this reason, structural assumptions, traditionally in terms of smoothness or sparsity, are often imposed on $\Theta$ to obtain uniformly consistent estimators of $\psi$ that converge to $\psi$ at parametric rates.


Recently, \citet{balakrishnan2026fundamental} introduced the \emph{structure-agnostic} (SA) models, a new class of submodels of $\calP$ that do not impose traditional structural assumptions on $\Theta$ such as smoothness or sparsity. The SA model, denoted as $\calP_{\SA} (\hat{\theta}, r_{n})$, is parameterized by a pair of indices $(\hat{\theta}, r_{n})$, where $\hat{\theta}$ is an initial estimator of $\theta$ treated as fixed and independent of the randomness of the data, and $r_{n}$ indicates convergence rates that are nonincreasing functions of $n$. Concretely, suppose that $\theta = (\theta_{1}, \cdots, \theta_{J})^{\top}$ and $r_{n} = (r_{n, 1}, \cdots, r_{n, J})^{\top}$ have $J$ components. Then $\calP_{\SA} (\hat{\theta}, r_{n})$ is generally defined as\footnote{In certain problems, we have $\theta_{j_{1}} \equiv \theta_{j_{2}}$ and $r_{n, j_{1}} \equiv r_{n, j_{2}}$. As is standard in the literature, in such a case, we use the same estimator $\hat{\theta}_{j_{1}} \equiv\hat{\theta}_{j_{2}}$ in computing $\hat{\psi}_{1, n}$. We will discuss the implications of this choice in the examples in Sections~\ref{sec:gsm}--~\ref{sec:ecc}; specifically, see Remarks~\ref{rem:cf1} and~\ref{rem:cf2}.}
\begin{equation}
\calP_{\SA} (\hat{\theta}, r_{n}) := \left\{  \bbP_{\theta} \in \calP: \Vert \hat{\theta}_{j} - \theta_{j} \Vert^{2} \leq r_{n, j}, j = 1, \cdots, J \right\}. \label{SA model}
\end{equation}
In other words, $\calP_{\SA} (\hat{\theta}, r_{n})$ contains the subset of all possible $\bbP_{\theta}$'s such that each component $\theta_{j}$ is contained in the corresponding $\sqrt{r_{n, j}}$-$\Vert \cdot \Vert$-neighborhood of a given initial estimator $\hat{\theta}_{j}$. For all the concrete examples in this paper (see Sections~\ref{sec:gsm}--\ref{sec:ecc}), we effectively take $r_{n, j}$ to be some large enough constant $R^{\ast}$ for $j > 2$ so assumptions are only imposed over at most two components of $\theta$.

The SA model $\calP_{\SA} (\hat{\theta}, r_{n})$ has several notable features. First, it does not impose explicit complexity reducing structural assumptions such as smoothness or sparsity\footnote{Even if the true smoothness or sparsity class were known, due to the theory-practice gap \citep{adcock2021gap, xu2022deepmed, chen2024causal}, when $\theta$ is estimated by modern deep neural networks, the SA model is still relevant because the smoothness/sparsity assumption alone fails to determine the properties of $\hat{\theta}$ or $\hat{\psi}_{1,n}$.}.
Second, \citet{balakrishnan2026fundamental} showed that the first-order DML estimator $\hat{\psi}_{1,n}$ of $\psi$ attains optimal convergence rates in the minimax sense under the SA model $\calP_{\SA} (\hat{\theta}, r_{n})$, for several concrete examples of $\psi$, including the quadratic functional in the Gaussian sequence model, the quadratic density integral functional (with the extra condition $r_{n}^{2} \gtrsim n^{-1}$ for these two examples; see Theorem~\ref{thm:high level} for explanation), and the expected conditional covariance. More recent follow-up papers \citep{jin2025structure, bonvini2024doubly, jin2025sharp, gu2026optimally} establish the minimaxity of $\hat{\psi}_{1,n}$ under the SA model for the average treatment effect, the average treatment effect on the treated, and other related parameters. Furthermore, the estimator $\hat{\psi}_{1, n}$ is the same and minimax in $\calP_{\SA} (\hat{\theta}, r_{n})$ for all values of $r_{n}$. Some analysts have used the minimaxity of $\hat{\psi}_{1, n}$ under the SA model $\calP_{\SA} (\hat{\theta}, r_{n})$ as a justification of the increasing use of $\hat{\psi}_{1, n}$ in (bio)statistics and econometrics.

\subsection*{Our Contribution}

Our main technical contribution of this paper is related to the decision-theoretic properties of estimators, tracing back to the classical work of Abraham Wald \citep{wald1941principles, wald1945statistical, wald1947essentially}. Wald is renowned as the inventor of the minimax principle \citep{wald1945statistical}, arguably the most popular theoretical paradigm used to measure the quality of an estimator by modern (bio)statisticians and econometricians \citep{brown1994minimaxity, andrews2021model, adusumilli2026sample} and adopted in \citet{balakrishnan2026fundamental}.

However, certain minimax estimators may be inadmissible. One celebrated example of the difference between minimaxity and (in)admissibility is Stein's paradox, asserting that the maximum likelihood estimator (MLE), although minimax, is everywhere dominated in mean squared error (MSE) for every sample size $n$ by the James-Stein (JS) estimator in the many-normal-means model of dimension at least three \citep{stein1956inadmissibility, james1961estimation, brown1971admissible}. Hence, the MLE is inadmissible. In this paper, we will show that, under $\calP_{\SA} (\hat{\theta}, r_{n})$, there exists a class of parameters $\psi$, which we refer to as the \emph{monotone bias class}, for which higher-order influence function (HOIF) estimators \citep{robins2008higher, robins2016technical, liu2017semiparametric} dominate the first-order DML estimator $\hat{\psi}_{1, n}$ in the large-$n$ limit whenever $\prod_{j = 1}^{J} r_{n, j} \gtrsim n^{-1}$ (see Theorem~\ref{thm:high level} for the precise statement). That is, we show that in the $\calP_{\SA} (\hat{\theta}, r_{n})$ model, the mimimax DML estimator is asymptotically inadmissible (in a uniform sense), when estimating a parameter in the \emph{monotone bias class}. To make the above claims precise, we next define: \emph{(uniformly) asymptotic (in)admissibility} and the \emph{monotone bias class} of functionals. We shall see that two out of the three functionals studied in \citet{balakrishnan2026fundamental} are in the \emph{monotone bias class}.

\begin{definition}[(Uniform) asymptotic (in)admissibility]\label{def:admissible} 
An estimator sequence $\{\psi_{n}\}$ of $\psi (\theta)$ indexed by $n$ is said to be (uniformly) \emph{asymptotically inadmissible} in scaled MSE loss under model $\calP_{\SA} (\hat{\theta}, r_{n})$ if there exists another estimator sequence $\{\psi_{n}'\}$ of $\psi (\theta)$ also indexed by $n$ such that the following hold. For every $\eps > 0$, there exists a sufficiently large natural number $N$, which may depend on $\eps$, such that as long as $n \geq N$,
\begin{equation}
\label{clauses}
\sup_{\bbP_{\theta} \in \calP_{\SA} (\hat{\theta}, r_{n})} \Delta_{n} (\psi_{n}', \psi_{n}; \theta) \leq \eps \text{ and } \inf_{\bbP_{\theta} \in \calP_{\SA} (\hat{\theta}, r_{n})} \Delta_{n} (\psi_{n}', \psi_{n}; \theta) < 0,
\end{equation}
where $\Delta_{n} (\psi_{n}', \psi_{n}; \theta) := \frac{\mse (\psi_{n}^{\prime}) - \mse (\psi_{n})}{\mse (\psi_{n})}$, $\mse (\cdot) \equiv \mse_{\theta} (\cdot) \coloneqq \E_{\theta} (\cdot - \psi (\theta))^{2}$. If otherwise, we say that $\{\psi_{n}\}$ is \emph{(uniformly) asymptotically admissible}.
\end{definition}

In the above definition, the difference in MSEs is scaled. Because the MSE of a reasonable estimator should decay to zero under the large $n$ limit, obtaining nontrivial results requires an appropriate scaling. Here, we choose the MSE of $\psi_{n}$ as a natural scaling factor. We refer the interested readers to Appendix~\ref{app:scale} for a more elaborate discussion on the choice of the denominator. Throughout the paper, we give a rigorous $\eps$-$\delta$ treatment to all technical results. In the text, however, we still use the convenient asymptotic notation such as $\gtrsim, \lesssim, \gg, \ll, o (\cdot), O (\cdot)$.


The following definition of the \emph{monotone bias class} is different from, but as explained below, is essentially equivalent to that in our previous work \citep{liu2020nearly}.

\begin{definition}
\label{def:mbc} 
Let $\bias (\hat{\psi}_{1, n})$ and $\var (\hat{\psi}_{1, n})$ be, respectively, the bias and variance of the first-order DML estimator $\hat{\psi}_{1, n}$ of $\psi$. $\psi$ is said to be in the \emph{monotone bias class} if there exists a sufficiently large natural number $N$ such that for every $\bbP_{\theta} \in \calP_{\SA} (\hat{\theta}, r_{n})$, as long as $n \geq N$, the following hold.
\begin{enumerate}
[label = (\arabic*)]
\item $\vert \bias (\hat{\psi}_{1, n}) \vert \leq C \prod_{j = 1}^{J} r_{n, j}$ for some universal constant $C > 0$ independent of $(\bbP_{\theta}, \hat{\theta}, r_{n})$, there always exists another estimator, denoted by $\hat{\psi}_{2, n}$, such that $\frac{|\bias (\hat{\psi}_{2, n})|}{|\bias (\hat{\psi}_{1, n})|} - 1 \leq 0$, and the inequality becomes strict at some law in $\calP_{\SA} (\hat{\theta}, r_{n})$;

\item The variances of $\hat{\psi}_{1, n}$ and $\hat{\psi}_{2, n}$ satisfy the following condition: there exists constants $v > 0$ and $C' > 0$ and a sequence $k$ with $k / n \leq C' \eps$ such that $\var (\hat{\psi}_{1, n}) = v / n$ and
\begin{equation}
\var (\hat{\psi}_{2, n}) - \var (\hat{\psi}_{1, n}) \leq C' \Big( \frac{\sqrt{|\bias (\hat{\psi}_{1, n})|}}{n} + \frac{k}{n^{2}} \Big). \label{var bound}
\end{equation}
\end{enumerate}
\end{definition}

Before moving forward, we briefly comment on the statistical implication of the above two definitions. The discussion here is in a similar spirit to Remark~1.7 of \citet{liu2020nearly}. Unlike textbook definition of (in)admissibility, we resort to asymptopia in Definition~\ref{def:admissible}. We mention in passing that \citet{read1972asymptotic} adopted a pointwise definition of asymptotic inadmissibility, whereas our definition is about uniform asymptotics. Definition~\ref{def:mbc} offers additional structures to the type of functionals, for which we can make sharp mathematical statement. Combined with Definition~\ref{def:admissible}, in the most pessimistic scenario (when $\hat{\psi}_{2, n}$ fails to reduce any bias), if one is willing to pay at most a small price $\eps > 0$ (depending on sample size $n$) due to the potentially slightly inflated variance of $\hat{\psi}_{2, n}$, one can gain a lot in certain optimistic cases in $\calP_{\SA} (\hat{\theta}, r_{n})$, in which $\hat{\psi}_{2, n}$ can have a much reduced bias. We are now ready to state the following general theorem regarding the asymptotic inadmissibility of the DML estimator $\hat{\psi}_{1, n}$, making the above argument rigorous. A proof is given following a few remarks afterwards. \setcounter{theorem}{-1}

\begin{theorem}
\label{thm:high level} 
For $\psi$ belonging to the \emph{monotone bias class}, under the SA model $\calP_{\SA} (\hat{\theta}, r_{n})$, an estimator $\hat{\psi}_{2, n}$ different from $\hat{\psi}_{1, n}$ satisfies the following property. For every $\eps > 0$, there exists a sufficiently large natural number $N$ that may depend on $\eps$, such that as long as $n \geq N$, (i) $\sup_{\bbP_{\theta} \in \calP_{\SA} (\hat{\theta}, r_{n})} \Delta_{n} (\hat{\psi}_{2, n}, \hat{\psi}_{1, n}; \theta) \leq \eps$ for any $r_{n}$, and (ii) $\inf_{\bbP_{\theta} \in \calP_{\SA} (\hat{\theta}, r_{n})} \Delta_{n} (\hat{\psi}_{2, n}, \hat{\psi}_{1, n}; \theta) < 0$ if $\prod_{j = 1}^{J} r_{n, j} \geq C / n$ for some absolute constant $C > 0$. Hence, in view of Definition~\ref{def:admissible}, $\hat{\psi}_{1, n}$ is asymptotically inadmissible in this case.
\end{theorem}

It should also be noted that Theorem~\ref{thm:high level} does not discuss minimaxity of either estimator. As we shall see, for two examples in the \emph{monotone bias class} -- the quadratic functional in the Gaussian sequence model in Section~\ref{sec:gsm} and the quadratic density integral functional in Section~\ref{sec:quad}, the minimaxity of $\hat{\psi}_{1, n}$ or $\hat{\psi}_{2, n}$ requires an additional condition $\sqrt{\prod_{j = 1}^{J} r_{n, j}} \gtrsim n^{-1}$. As explained in \citet{balakrishnan2026fundamental}, when $\sqrt{\prod_{j = 1}^{J} r_{n, j}} \ll n^{-1}$, the so-called plug-in estimators dominate both $\hat{\psi}_{1, n}$ and $\hat{\psi}_{2, n}$ in these two examples, because the plug-in estimators have zero variance and squared bias $\sqrt{\prod_{j = 1}^{J} r_{n, j}} \ll n^{-1}$, while both $\hat{\psi}_{1, n}$ and $\hat{\psi}_{2, n}$ have variances of order $n^{-1}$. However, as also noted by \citet{balakrishnan2026fundamental}, $\sqrt{\prod_{j = 1}^{J} r_{n, j}} \ll n^{-1}$ generally does not hold if the sample size used to estimate $\hat{\theta}$ is of the same order as $n$. Therefore, there is essentially no loss of generality if we exclude the case where $\sqrt{\prod_{j = 1}^{J} r_{n, j}} \ll n^{-1}$ holds, as we do in Sections~\ref{sec:gsm} and \ref{sec:quad}, in which case both $\hat{\psi}_{1, n}$ and $\hat{\psi}_{2, n}$ are minimax rate-optimal.

\begin{remark}
\label{rem:high}
In all the examples covered in this paper and in \citet{balakrishnan2026fundamental}, $|\bias (\hat{\psi}_{1, n})|^{2}$ is upper bounded by $\prod_{j = 1}^{J} r_{n, j}$ times a constant. In view of Theorem~\ref{thm:high level}, when $\psi$ is in the \emph{monotone bias class}, $\hat{\psi}_{1, n}$ is \emph{asymptotically inadmissible} and is uniformly asymptotically dominated by $\hat{\psi}_{2, n}$ when $|\bias (\hat{\psi}_{1, n})| \gtrsim n^{-1 / 2}$. When $|\bias (\hat{\psi}_{1, n})| \ll n^{-1/2}$, $\hat{\psi}_{2, n}$ is asymptotically never worse than $\hat{\psi}_{1, n}$ and both are minimax. In particular, when DML estimators $\hat{\psi}_{1, n}$ are deployed in practice, it is often implicitly assumed that $\prod_{j = 1}^{J} r_{n, j} \ll n^{-1}$ holds, a condition often referred to as the \emph{rate-double-robustness} when $J = 2$. Under this condition, neither estimator dominates the other; further, $|\bias (\hat{\psi}_{1, n})| \ll n^{-1 / 2}$ automatically holds, and thus both $\hat{\psi}_{1, n}$ and $\hat{\psi}_{2, n}$ can be used to construct an asymptotically valid Wald confidence interval (CI) of length $O (n^{-1 / 2})$, which is common practice, although other non-Wald CI constructions are also under rapid development \citep{wasserman2020universal, zheng2025perturbed}.
\end{remark}

Since in practice $r_{n}$ is unknown and the possibility that it is of order $1$ cannot be empirically excluded \citep{robins1997toward, ritov2014bayesian}, we consider the following assumption-lean model.
\begin{definition}
\label{def:al} 
Given a constant $R^{\ast} > 0$, we define the model $\calP_{\AL} (\hat{\theta}) \equiv \calP_{\SA} (\hat{\theta}, r_{n} = (R^{\ast}, \cdots, R^{\ast}))$ as the assumption-lean model \citep{liu2024assumption}.
\end{definition}

Under the assumption-lean model, for parameters in the \emph{monotone bias class}, it follows from Theorem~\ref{thm:high level} that $\hat{\psi}_{2, n}$ dominates $\hat{\psi}_{1, n}$ asymptotically, because $\prod_{j = 1}^{J} r_{n, j} = 1 \gg n^{-1}$, and thus $\hat{\psi}_{1, n}$ is asymptotically inadmissible.

In fact, as discussed in Section~\ref{sec:sa} below, we can say more. Given a parameter $\psi$ in the \emph{monotone bias class}, we can construct an asymptotically level-$\alpha$ falsification test of the null hypothesis $\calH_{0}$: $\bias (\hat{\psi}_{1, n}) \ll n^{-1/2}$ that, when $\calH_{0}$ is rejected, provides empirical evidence for the alternative hypothesis $|\bias (\hat{\psi}_{1, n})| \gtrsim n^{-1/2}$ and thus also empirical evidence that the Wald CI centered on $\hat{\psi}_{1, n}$ under-covers even in large samples \citep{liu2020nearly}. However, by the aforementioned results of \citet{robins1997toward} and \citet{ritov2014bayesian}, any such test must be inconsistent under the assumption-lean model $\calP_{\AL} (\hat{\theta})$. Hence, failure to reject does not provide evidence for or against the null hypothesis $\calH_{0}$ even asymptotically. We return to this issue in the concluding section of the paper (Section~\ref{sec:sa}).

\begin{remark}
\label{rem:intuition}
It will be clear in later sections that all examples studied by \citet{balakrishnan2026fundamental} are in the \emph{monotone bias class} defined in Definition~\ref{def:mbc}, except for the expected conditional covariance (see Section~\ref{sec:ecc}). The two parts of conditions in Definition~\ref{def:mbc} need further elaboration. Part (1) states that as long as $n$ surpasses a universal threshold $N$, the bias of $\hat{\psi}_{2, n}$ is never greater than but sometimes strictly smaller than that of $\hat{\psi}_{1, n}$. Part (2) says that the difference between the variances of $\hat{\psi}_{1, n}$ and $\hat{\psi}_{2, n}$ is essentially of order $O (n^{-1})$ with the hidden constant scaled with the root bias of $\hat{\psi}_{1, n}$. The original definition of the \emph{monotone bias class} in \citet{liu2020nearly} contains only part (1) but part (2) was implicit. As described later, the alternative estimator $\hat{\psi}_{2, n}$ is a second-order $U$-statistic (heretofore referred to as the second-order estimators for simplicity), constructed via the theory of HOIFs \citep{robins2008higher, robins2016technical, liu2017semiparametric}. Such second-order estimators satisfy both Parts (1) and (2). The difference between the variances of $\hat{\psi}_{2, n}$ and $\hat{\psi}_{1, n}$ can be bounded as follows, as is the case in all our examples in later sections:
\begin{align*}
\var (\hat{\psi}_{2, n}) - \var (\hat{\psi}_{1, n}) \lesssim \frac{k}{n^{2}} + \frac{|\bias (\hat{\psi}_{1, n})| - |\bias (\hat{\psi}_{2, n})|}{n} + \frac{v^{1 / 2} \{|\bias (\hat{\psi}_{1, n})| - |\bias (\hat{\psi}_{2, n})|\}^{1 / 2}}{n},
\end{align*}
where $k$ is a tuning parameter chosen so that $k = o (n)$. It is then not difficult to verify that Part (2) of Definition~\ref{def:mbc} holds. In terms of Part (1), the second-order estimator $\hat{\psi}_{2, n}$ can be viewed as debiasing $\hat{\psi}_{1, n}$ by estimating a part of $\bias (\hat{\psi}_{1, n})$; also see comments after Lemma~\ref{lem:quad gsm 2}, \ref{lem:quad2}, and \ref{lem:ecc2}, and Theorem~\ref{thm:ecv}. For functionals outside the \emph{monotone bias class} such as the expected conditional covariance functional covered in Section~\ref{sec:ecc}, $\hat{\psi}_{2, n}$ is still rate-optimal but may have a bias exceed that of $\hat{\psi}_{1, n}$ under certain laws in $\calP_{\SA} (\hat{\theta}, r_{n})$ .
\end{remark}

With these ingredients, the proof of Theorem~\ref{thm:high level} is almost immediate by elementary calculations, so we record the proof here.

\begin{proof}[Proof of Theorem~\ref{thm:high level}]
Take $N$ to be the large enough natural number in Definition~\ref{def:mbc}, from which the result follows directly. To see this, we first decompose the scaled MSE difference as
\begin{align*}
\frac{\mse (\hat{\psi}_{2, n}) - \mse (\hat{\psi}_{1, n})}{\mse (\hat{\psi}_{1, n})} = - \frac{\bias (\hat{\psi}_{1, n})^{2} - \bias^{2} (\hat{\psi}_{2, n})}{\bias (\hat{\psi}_{1, n})^{2} + \var (\hat{\psi}_{1, n})} + \frac{\var (\hat{\psi}_{2, n}) - \var (\hat{\psi}_{1, n})}{\bias (\hat{\psi}_{1, n})^{2} + \var (\hat{\psi}_{1, n})}.
\end{align*}
Let $n \geq N$. Then we always have $\bias (\hat{\psi}_{1, n})^{2} - \bias^{2} (\hat{\psi}_{2, n}) \geq 0$ by Definition~\ref{def:mbc}(1). Next, by Definition~\ref{def:mbc}(2),
\begin{align*}
& \ \frac{\mse (\hat{\psi}_{2, n}) - \mse (\hat{\psi}_{1, n})}{\bias (\hat{\psi}_{1, n})^{2} + \var (\hat{\psi}_{1, n})} \leq \frac{\sqrt{\bias (\hat{\psi}_{1, n})} + \frac{k}{n}}{n \bias (\hat{\psi}_{1, n})^{2} + v} \leq \eps.
\end{align*}

When $\prod_{j = 1}^{J} r_{n, j} \geq C / n$ for some constant $C > 0$, there must exist a law such that $\bias (\hat{\psi}_{1, n}) \geq \sqrt{C / n}$ and $\bias (\hat{\psi}_{2, n}) = 0$. Then we have
\begin{align*}
\frac{\mse (\hat{\psi}_{2, n}) - \mse (\hat{\psi}_{1, n})}{\bias (\hat{\psi}_{1, n})^{2} + \var (\hat{\psi}_{1, n})} & \leq \frac{\sqrt{\bias (\hat{\psi}_{1, n})} + \frac{k}{n} - n \bias (\hat{\psi}_{1, n})^{2}}{n \bias (\hat{\psi}_{1, n})^{2} + v} = \frac{\sqrt{\bias (\hat{\psi}_{1, n})} + \frac{k}{n} + v}{n \bias (\hat{\psi}_{1, n})^{2} + v} - 1 \\
& \leq \frac{\sqrt{\bias (\hat{\psi}_{1, n})}}{n \bias (\hat{\psi}_{1, n})^{2} + v} + \frac{\frac{k}{n} + v}{C + v} - 1 < 0,
\end{align*}
which completes the proof.
\end{proof}

\section*{Notation and Organization}

Throughout this paper, we always let $C > 0$ denote a sufficiently large constant independent of the sample size $n$. For any functions mentioned in the paper, they are understood to be squared integrable with respect to the Lebesgue measure. $\bbU_{n, m} [\cdot]$ denotes a $m$-th order $U$-statistic operator. Given a collection of $k$ different functions $\bar{f}_{k} = (f_{1}, \cdots, f_{k})^{\top}$, we let $\mathsf{\Pi}_{\bbP} (\cdot \mid \bar{f}_{k})$ denote the operator of $L^{2} (\bbP)$-projection onto the linear span of $\bar{f}_{k}$ and $\Vert \cdot \Vert_{2, \bbP}$ denote the $L^{2} (\bbP)$-norm. If $\bbP$ is the Lebesgue measure, we omit the subscript and write $\mathsf{\Pi} (\cdot \mid \bar{f}_{k})$ and $\Vert \cdot \Vert_{2}$ for short. $\Vert\cdot\Vert_{2}$ also denotes the $\ell^{2}$-norm of a vector. We denote the population Gram matrix of $\bar{f}_{k}$ under the distribution $\bbP$ as $\Sigma_{\bbP, \bar{f}_{k}} \coloneqq \E [\bar{f}_{k} (X) \bar{f}_{k} (X)^{\top}]$. When it is clear from the context, we omit the dependence in the subscript on $\bbP$ or $\bar{f}_{k}$ or both.

The remainder of this paper makes Theorem~\ref{thm:high level} concrete. Sections~\ref{sec:gsm}--\ref{sec:ecc} cover four examples of $\psi$, one of which does not belong to the \emph{monotone bias class}. Theorem~\ref{thm:high level} will then be specialized for the three examples in the \emph{monotone bias class}. Section~\ref{sec:sa} concludes the paper by making some additional comments on the relevance of the SA model to practitioners who are more interested in uncertainty quantification or statistical inference. Proofs are deferred to the Appendix.

\section{Quadratic Functional in the Gaussian Sequence Model}
\label{sec:gsm}

As in \citet{balakrishnan2026fundamental}, we observe data drawn from the infinite Gaussian sequence model:
\begin{align}
\label{gsm}
Y_{i} = \theta_{i} + \varepsilon_{i}, i = 1, 2, \cdots
\end{align}
where $\{\varepsilon_{i}, i = 1, 2, \cdots\} \overset{\mathrm{i.i.d.}}{\sim} \mathrm{N} (0, n^{-1})$. Let $\theta := \{\theta_{i}, i = 1, 2, \cdots\}$ and $Y := \{Y_{i}, i = 1, 2, \cdots\}$. We are interested in learning about the quadratic functional
\begin{align}
\label{gsm:quad}
Q (\theta) := \Vert \theta \Vert_{2}^{2} \equiv \sum_{i = 1}^{\infty} \theta_{i}^{2}.
\end{align}

The SA model corresponding to $\psi(\theta)$ is defined by \citet{balakrishnan2026fundamental} as:
\begin{align}
\label{quad gsm model}
\calP_{\SA} ((\hat{\theta}, \hat{\theta}), (r_{n}, r_{n})) \coloneqq \left\{ \bbP_{\theta}: \Vert \hat{\theta} - \theta \Vert_{2}^{2} \leq r_{n} \right\} ,
\end{align}
where $\hat{\theta}$ is some initial estimator of $\theta$. It is noteworthy that we deliberately write $\hat{\theta}$ and $r_{n}$ twice in the notation $\calP_{\SA} ((\hat{\theta}, \hat{\theta}), (r_{n}, r_{n}))$ to emphasize that we take $J = 2$ and $\prod_{j = 1}^{J} r_{n, j} = r_{n}^{2}$ in this case. In the sequel, however, we write $\calP_{\SA} (\hat{\theta}, r_{n})$ instead in the text to simplify the notation. We adopt a similar convention for the quadratic density integral functional in Section~\ref{sec:quad} and the expected conditional variance in Section~\ref{sec:ecc}.

\citet{balakrishnan2026fundamental} obtained the following results.

\begin{lemma}
\label{lem:quad gsm} 
The following hold. Suppose that there exists a universal constant $C > 0$ such that $r_{n} \geq C / n$,
\begin{align*}
\mathfrak{R}_{n} (Q; \calP_{\SA} (\hat{\theta}, r_{n})) \coloneqq \inf_{\hat{Q}} \sup_{\bbP_{\theta} \in\calP_{\SA} (\hat{\theta}, r_{n})} \E_{\theta} [(\hat{Q} - Q (\theta))^{2}] \geq C \Big( r_{n}^{2} + \frac{\Vert \hat{\theta} \Vert_{2}^{2}}{n} \Big).
\end{align*}
This lower bound is attained by the first-order estimator $\hat{Q}_{1, n}$, defined as
\begin{align*}
\hat{Q}_{1, n} \coloneqq 2 \langle Y, \hat{\theta} \rangle - \Vert\hat{\theta} \Vert_{2}^{2}.
\end{align*}
The bias, variance, and mean squared error (MSE) of $\hat{Q}_{1, n}$ have the following forms:
\begin{align*}
\bias (\hat{Q}_{1, n}) = - \Vert\hat{\theta} - \theta\Vert_{2}^{2}, \var (\hat{Q}_{1, n}) = \frac{4}{n} \Vert \hat{\theta} \Vert_{2}^{2}, \text{ and } \mse (\hat{Q}_{1, n}) = \Vert \hat{\theta} - \theta \Vert_{2}^{4} + \frac{4}{n} \Vert \hat{\theta} \Vert_{2}^{2}.
\end{align*}
\end{lemma}

The lower and upper bounds can be found in Theorem~1, Part~1 and Theorem~2, Part~1 of \citet{balakrishnan2026fundamental}, respectively. These results, taken together, prove the rate optimality of $\hat{Q}_{1, n}$ in the minimax sense.

To show the asymptotic inadmissibility of the minimax estimator $\hat{Q}_{1, n}$, we need to exhibit a different estimator that improves upon $\hat{Q}_{1, n}$. To this end, we adopt the following second-order estimator appeared in \cite{robins2006adaptive}:
\begin{align*}
\hat{Q}_{2, n} (k) & := \sum_{i = 1}^{k} Y_{i, 1} Y_{i, 2} + \sum_{i = k+1}^{\infty} (2 Y_{i}\hat{\theta}_{i} - \hat{\theta}_{i}^{2}) \\
& \equiv \hat{Q}_{1, n} + \sum_{i = 1}^{k} (Y_{i, 1} Y_{i, 2} - 2 Y_{i} \hat{\theta}_{i} + \hat{\theta}_{i}^{2}),
\end{align*}
where $Y_{i, 1} \coloneqq Y_{i} + \Phi^{-1} (U_{i}) / \sqrt{n}$, $Y_{i,2} \coloneqq Y_{i} - \Phi^{-1} (U_{i}) / \sqrt{n}$, $\Phi$ is the standard normal cumulative distribution function and $U_{i}$'s are independent uniform random variables over $[0,1]$. Here $Y_{i, 1} \mathpalette{\protect\independenT}{\perp} Y_{i, 2}$. The difference between $\hat{Q}_{2, n}(k)$ and $\hat{Q}_{1, n}$ is an unbiased estimator of $\Vert \mathsf{\Pi}_{k} (\theta - \hat{\theta}) \Vert_{2}^{2}$, where $\mathsf{\Pi}_{k} (\cdot)$ denotes the projection onto the first $k$ coordinates of the input infinite-dimensional vector, with $\mathsf{\Pi}_{k}^{\perp} (\cdot)$ naturally meaning the projection onto the $(k+1)$-th coordinate and onward.

The following lemma characterizes the bias, variance, and mean squared error of $\hat{Q}_{2, n} (k)$. A proof can be found in Appendix~\ref{app:lem:quad gsm 2}.

\begin{lemma}
\label{lem:quad gsm 2} 
The bias and variance of $\hat{Q}_{2, n} (k)$ under $\bbP_{\theta} \in \calP_{\SA} (\hat{\theta}, r_{n})$ are characterized as follows:
\begin{align*}
\bias (\hat{Q}_{2, n} (k)) & = - \Vert \mathsf{\Pi}_{k}^{\perp} (\hat{\theta} - \theta) \Vert_{2}^{2} \lesssim r_{n}, \\
\var (\hat{Q}_{2, n} (k)) & = \frac{4}{n} \Vert \hat{\theta} \Vert_{2}^{2} + \frac{4 k}{n^{2}} + \frac{4}{n} \Vert \mathsf{\Pi}_{k} (\hat{\theta} - \theta) \Vert_{2}^{2} - \frac{4}{n} \langle \mathsf{\Pi}_{k} \hat{\theta}, \mathsf{\Pi}_{k} (\hat{\theta} - \theta) \rangle \lesssim \frac{1}{n},
\end{align*}
where the inequalities hold for $\bbP_{\theta} \in \calP_{\SA} (\hat{\theta}, r_{n})$.
\end{lemma}

We note that $\Vert \mathsf{\Pi}_{k} (\hat{\theta} - \theta) \Vert_{2}^{2} \equiv \bias (\hat{Q}_{1, n}) - \bias (\hat{Q}_{2, n} (k))$ so $\hat{Q}_{2, n} (k)$ corrects the bias of $\hat{Q}_{1, n}$ by estimating a lower bound of $\bias (\hat{Q}_{1, n}) \lesssim r_{n}$. By Lemma~\ref{lem:quad gsm 2}, $Q (\theta)$ belongs to the \emph{monotone bias class}. Comparing $\mse (\hat{Q}_{2, n} (k))$ and $\mse (\hat{Q}_{1, n})$ in the asymptotic sense, we obtain the first main statistical result of this paper. There always exists a distribution in $\calP_{\SA} (\hat{\theta}, r_{n})$ such that Definition~\ref{def:mbc}(1) holds. To see this, consider the case where $\theta = \hat{\theta} + r_{n}^{1 / 2} \upsilon$, where $\Vert \upsilon \Vert_{2} = 1$ and the coordinates of $\upsilon$ from $k + 1$ onward are all zeros. With this choice, $\bias (\hat{Q}_{2, n} (\bar{\phi}_{k})) = 0$. The rest of the proof can be found in Appendix~\ref{app:quad gsm}.

\begin{theorem}
\label{thm:quad gsm} 
Choosing $k = o (n)$. Under Model $\calP_{\SA} (\hat{\theta}, r_{n})$ with $r_{n} \geq C / n$ for some absolute constant $C > 0$, $\hat{Q}_{2, n} (k)$ is \emph{asymptotically minimax}. The same conclusions as in Theorem~\ref{thm:high level} hold by replacing $\psi (\theta)$ as $Q (\theta)$, $\hat{\psi}_{1, n}$ as $\hat{Q}_{1, n}$, and $\hat{\psi}_{2, n}$ as $\hat{Q}_{2, n} (k)$. The same conclusions also hold when we replace the SA model $\calP_{\SA} (\hat{\theta}, r_{n})$ with the assumption-lean model $\calP_{\AL} (\hat{\theta})$ and drop the assumptions on $r_{n}$.
\end{theorem}

Echoing the comment right after Theorem~\ref{thm:high level}, for the minimaxity of $\hat{Q}_{1, n}$ or $\hat{Q}_{2, n} (k)$, we need $r_{n} \gtrsim n^{-1}$. When $r_{n} \ll n^{-1}$, the so-called plug-in estimator $\hat{Q}_{\mathrm{pi}} \coloneqq \Vert\hat{\theta} \Vert^{2}$ has zero variance and squared bias of order $r_{n}$, thus dominating both $\hat{Q}_{1, n}$ and $\hat{Q}_{2, n} (k)$ when $\Vert\hat{\theta} \Vert^{2}$ is of order 1. However, $r_{n} \ll n^{-1}$, or equivalently $\Vert \hat{\theta} - \theta \Vert\ll n^{-1 / 2}$, is generally difficult to hold if the sample used to compute $\hat{\theta}$ is of size similar to $n$, as such a condition says that we can estimate the possibly infinite-dimensional $\theta$ at a rate much faster than the parametric rate. A similar discussion also applies to the quadratic density integral functional to be discussed next.

\section{Quadratic Density Integral Functional}
\label{sec:quad}

The second example is about estimating the quadratic density integral functional of the probability density function $f$ of $X$ based on $n$ i.i.d. observations $\{X_{i} \in [0, 1]^{d}\}_{i = 1}^{n} \sim f$:
\begin{equation}
\label{quad}
\psi(f) \coloneqq \int f (x)^{2} \diff x.
\end{equation}

The SA model corresponding to $\psi(f)$ is defined by \citet{balakrishnan2026fundamental} as:
\begin{align}
\label{quad model}
\calP_{\SA} ((\hat{f}, \hat{f}), (r_{n}, r_{n})) \coloneqq \left\{ \bbP_{f}: \Vert \hat{f} - f \Vert_{2}^{2} \leq r_{n}, \int f (x) \diff x = 1, f \geq 0, \Vert \hat{f} \Vert_{\infty} \leq C, \Vert f \Vert_{\infty} \leq C \right\}  ,
\end{align}
where $\hat{f}$ is some initial estimator of $f$ computed from a separate independent sample treated as fixed. Similar to the case in Section~\ref{sec:gsm}, we take $J = 2$ and $\prod_{j = 1}^{J} r_{n, j} = r_{n}^{2}$, and write $\calP_{\SA} (\hat{f}, r_{n})$ instead.

\begin{remark}
\label{rem:cf1} 
As mentioned in footnote 1, we use the same estimator $\hat{f} \equiv \hat{f}_{1} \equiv \hat{f}_{2}$ to compute $\hat{\psi}_{1, n}$, which is the standard DML estimator commonly employed in the literature \citep{chernozhukov2018double} but excludes more refined estimators with $f$ estimated by separate samples studied in \citet{newey2018cross, mcgrath2026nuisance, mcclean2026double}.
\end{remark}

The following lemma, paraphrasing the results of \citet{balakrishnan2026fundamental}, summarizes the lower and upper bounds of the error rate of estimating $\psi (f)$ under $\calP_{\SA} (\hat{f}, r_{n})$.

\begin{lemma}
\label{lem:quad} 
The following hold. Suppose that there exists a universal constant $C > 0$ such that $r_{n} \geq C / n$,
\begin{align*}
\mathfrak{R}_{n} (\psi; \calP_{\SA} (\hat{f}, r_{n})) \coloneqq \inf_{\hat{\psi}} \sup_{\bbP_{f} \in \calP_{\SA} (\hat{f}, r_{n})} \E_{f} [(\hat{\psi} - \psi(f))^{2}] \geq C \Big\{ r_{n}^{2} + \frac{1}{n} \Big( \Vert \hat{f} \Vert_{3}^{3} - \Vert \hat{f} \Vert_{2}^{4} \Big) \Big\}.
\end{align*}
This lower bound is attained by the first-order estimator $\hat{\psi}_{1, n}$, defined as
\begin{align*}
\hat{\psi}_{1, n} \coloneqq \frac{2}{n} \sum_{i = 1}^{n} \hat{f} (X_{i}) - \int \hat{f} (x)^{2} \diff x.
\end{align*}
The bias, variance and MSE of $\hat{\psi}_{1, n}$ have the following forms:
\begin{align*}
\bias (\hat{\psi}_{1, n}) & = - \int (\hat{f} (x) - f (x))^{2} \diff x \equiv - \Vert \hat{f} - f \Vert_{2}^{2}, \\
\var (\hat{\psi}_{1, n}) & = \frac{4}{n} \var (\hat{f} (X)) \equiv \frac{4}{n} \Big\{ \int \hat{f} (x)^{2} f (x) \diff x - \Big( \int \hat{f} (x) f (x) \diff x \Big)^{2} \Big\}, \text{ and} \\
\mse (\hat{\psi}_{1, n}) & = \Vert \hat{f} - f \Vert_{2}^{4} + \frac{4}{n} \var (\hat{f} (X)).
\end{align*}
\end{lemma}

The lower and upper bounds can be found in Theorem~1, Part~2 and Theorem~2, Part~2 of \citet{balakrishnan2026fundamental}, respectively. These results, taken together, prove the optimality of $\hat{\psi}_{1, n}$ in the minimax sense.

To show the asymptotic inadmissibility of the minimax estimator $\hat{\psi}_{1, n}$, when $r_{n}^{2} \gtrsim n^{-1}$ or equivalently $r_{n}^{1 / 2} \gtrsim n^{-1 / 4}$, we exhibit a different estimator that improves on $\hat{\psi}_{1, n}$. To this end, let $\bar{\phi}_{k} \coloneqq (\phi_{1}, \cdots, \phi_{k})^{\top}$ be a $k$-dimensional orthonormal basis with respect to the Lebesgue measure \citep{chen2007large}. We then construct the following second-order $U$-statistic estimator:
\begin{align*}
\hat{\psi}_{2, n} (\bar{\phi}_{k}) & := \hat{\psi}_{1, n} + \bbU_{n, 2} \left[ \left( \bar{\phi}_{k} (X_{1}) - \int \bar{\phi}_{k} (x) \hat{f} (x) \diff x \right)^{\top} \left( \bar{\phi}_{k} (X_{2}) - \int \bar{\phi}_{k} (x) \hat{f} (x) \diff x \right) \right] \\
& \equiv \bbU_{n, 1} [2 (\hat{f} (X) - \mathsf{\Pi} [\hat{f} \mid \bar{\phi}_{k}] (X))] + \bbU_{n, 2} [\bar{\phi}_{k} (X_{1})^{\top} \bar{\phi}_{k} (X_{2})] - \int \left( \hat{f} (x)^{2} - \mathsf{\Pi}[\hat{f} \mid \bar{\phi}_{k}] (x)^{2} \right) \diff x.
\end{align*}

\begin{remark}
\label{rem:quad2} 
Expert readers shall realize that $\hat{\psi}_{2, n} (\bar{\phi}_{k})$ debiases $\hat{\psi}_{1, n}$ by subtracting an unbiased estimator of a part of its bias, based on HOIFs. The part of the bias of $\hat{\psi}_{1, n}$ to be estimated is determined by the choice of $\bar{\phi}_{k}$. We mention in passing that the falsification test of \citet{liu2020nearly} mentioned earlier is essentially based on the statistic $\hat{\psi}_{2, n} (\bar{\phi}_{k}) - \hat{\psi}_{1, n}$. Similar tests or estimators have also been considered in instrumental variable or proximal causal inference settings \citep{breunig2024adaptive, liu2024assumption}.
\end{remark}

Let $\eta := \int \bar{\phi}_{k} (x) f (x) \diff x$ and $\hat{\eta} := \int \bar{\phi}_{k} (x) \hat{f} (x) \diff x$. We also make the following assumption on $\Sigma$.

\begin{assumption}
\label{as:Sigma quad} 
$\Sigma$ is assumed to have bounded spectra.
\end{assumption}

We now state the following lemma. The proof can be found in Appendix~\ref{app:lem:quad2}.

\begin{lemma}
\label{lem:quad2} 
The bias and variance of $\hat{\psi}_{2, n} (\bar{\phi}_{k})$ under $\bbP_{f} \in \calP_{\SA} (\hat{f}, r_{n})$ are characterized as follows:
\begin{align*}
\bias (\hat{\psi}_{2, n} (\bar{\phi}_{k})) = & - \int (\hat{f} (x) - f (x))^{2} \diff x + \int \mathsf{\Pi}[\hat{f} - f \mid \bar{\phi}_{k}] (x)^{2} \diff x \equiv - \Vert \hat{f} - f \Vert_{2}^{2} + \Vert \mathsf{\Pi} [\hat{f} - f \mid \bar{\phi}_{k}] \Vert_{2}^{2} \\
\equiv & - \Vert\hat{f} - f \Vert_{2}^{2} + \Vert \hat{\eta} - \eta \Vert_{2}^{2} \equiv - \Vert \mathsf{\Pi}^{\perp} [\hat{f} - f \mid \bar{\phi}_{k}] \Vert_{2}^{2} \lesssim r_{n},\\
\var (\hat{\psi}_{2, n} (\bar{\phi}_{k})) = & \ \frac{4}{n} \var [\hat{f} (X)] + \frac{8}{n} \left( \int\bar{\phi}_{k} (x) f (x) \hat{f} (x) \diff x - \hat{\eta} \right)^{\top} (\hat{\eta} - \eta) \\
& + \frac{2}{n (n - 1)} \left\{ \begin{array}[c]{c} 
\Tr (\Sigma^{2}) - 4 \hat{\eta}^{\top} \Sigma \eta + 2 \hat{\eta}^{\top} \Sigma \hat{\eta} + 2 \hat{\eta}^{\top} \hat{\eta} \cdot \eta^{\top} \eta \\
+ \, 2 (\hat{\eta}^{\top} \eta)^{2} - 4 \hat{\eta}^{\top} \hat{\eta} \cdot \hat{\eta}^{\top} \eta+ (\hat{\eta}^{\top} \hat{\eta})^{2}
\end{array} \right\} \\
\leq & \ \frac{4}{n} \var [\hat{f} (X)] + \frac{C}{n} \Vert \mathsf{\Pi} [\hat{f} - f \mid \bar{\phi}_{k}] \Vert_{2} + \frac{C k}{n^{2}} \lesssim \frac{1}{n},
\end{align*}
where the inequalities hold for $\bbP_{f} \in \calP_{\SA} (\hat{f}, r_{n})$ under Assumption~\ref{as:Sigma quad}.
\end{lemma}

We note that $\Vert\mathsf{\Pi} [\hat{f} - f \mid \bar{\phi}_{k}] \Vert_{2}^{2} \equiv \bias (\hat{\psi}_{1, n}) - \bias (\hat{\psi}_{2, n} (\bar{\phi}_{k}))$ so $\hat{\psi}_{2, n} (\bar{\phi}_{k})$ corrects the bias of $\hat{\psi}_{1, n}$ by estimating a lower bound of $\bias (\hat{\psi}_{1, n}) \lesssim r_{n}$. By Lemma~\ref{lem:quad gsm}, $\psi (f)$ belongs to the \emph{monotone bias class}. The following theorem therefore instantiates Theorem~\ref{thm:high level} for the quadratic density integral functional $\psi (f)$. There always exists a distribution in $\calP_{\SA} (\hat{f}, r_{n})$ such that Definition~\ref{def:mbc}(1) holds. To see this, consider the case where $f = \hat{f} + r_{n}^{1 / 2} \beta^{\top} \bar{\phi}_{k}$ with $\Vert \beta \Vert_{2} = 1$, for which $\bias (\hat{\psi}_{2, n} (\bar{\phi}_{k})) = 0$. The rest of the proof can be found in Appendix~\ref{app:quad}.

\begin{theorem}
\label{thm:quad} 
Choosing $k = o (n)$. Under model $\calP_{\SA} (\hat{f}, r_{n})$ and Assumption~\ref{as:Sigma quad}, $\hat{\psi}_{2, n} (\bar{\phi}_{k})$ is \emph{asymptotically minimax}. The same conclusions as in Theorem~\ref{thm:high level} hold by replacing $\psi (\theta)$ as $\psi (f)$, $\hat{\psi}_{1, n}$ as $\hat{\psi}_{1, n}$, and $\hat{\psi}_{2, n}$ as $\hat{\psi}_{2, n} (\bar{\phi}_{k})$. The same conclusions also hold when we replace the SA model $\calP_{\SA} (\hat{f}, r_{n})$ with the assumption-lean model $\calP_{\AL} (\hat{f})$ and drop the assumptions on $r_{n}$.
\end{theorem}

\section{Expected Conditional Covariance}
\label{sec:ecc}

All the functionals that we have analyzed so far fall within the \emph{monotone bias class}. In this section, we turn to the Expected Conditional Covariance (ECC) functional, defined as
\begin{align*}
\psi (a, b) \coloneqq \E [\cov (A, Y \mid X)] \equiv \E [(A - a (X)) (Y - b (X))],
\end{align*}
where $X \in [0, 1]^{d}$ denotes the baseline covariates, $A, Y \in \bbR$ are two types of responses, $a (\cdot) \coloneqq \E (A \mid X = \cdot)$ and $b (\cdot) \coloneqq \E (Y \mid X = \cdot)$. The ECC functional, as extensively discussed in \citet{liu2020nearly}, is not in the \emph{monotone bias class}. Therefore, not surprisingly, we can no longer conclude the asymptotic inadmissibility of the first-order DML estimator $\hat{\psi}_{1, n}$ for $\psi (a, b)$. Specifically, based on $n$ i.i.d. observations $\{X_{i}, A_{i}, Y_{i}\}_{i = 1}^{n}$, $\hat{\psi}_{1, n}$ reads as:
\begin{align*}
\hat{\psi}_{1, n} \coloneqq \frac{1}{n} \sum_{i = 1}^{n} (A_{i} - \hat{a} (X_{i})) (Y_{i} - \hat{b} (X_{i})).
\end{align*}

As usual, before presenting our new results, we first summarize the statistical properties and minimaxity of $\hat{\psi}_{1, n}$ obtained in \citet{balakrishnan2026fundamental} under the SA model defined by \citet{balakrishnan2026fundamental} for $\psi(a, b)$:
\begin{equation}
\label{ecc}
\calP_{\SA} ((\hat{a}, \hat{b}), (r_{n}, s_{n})) \coloneqq \left\{ \bbP_{a, b}: \Vert a - \hat{a} \Vert_{2, \bbP}^{2} \lesssim r_{n}, \Vert b - \hat{b} \Vert_{2, \bbP}^{2} \lesssim s_{n} \right\}.
\end{equation}
We let $p$ denote the marginal density of $X$, which, for simplicity, is assumed to be $\mathrm{Unif} ([0, 1]^{d})$.

\begin{lemma}
\label{lem:ecc} 
The following hold. There exists a constant $C > 0$ such that
\begin{align*}
\mathfrak{R}_{n} (\psi; \calP_{\SA} ((\hat{a}, \hat{b}), (r_{n}, s_{n}))) \coloneqq \inf_{\hat{\psi}} \sup_{\bbP_{a, b} \in \calP_{\SA} ((\hat{a}, \hat{b}), (r_{n}, s_{n}))} \E_{a, b} [(\hat{\psi} - \psi(a, b))^{2}] \geq C \Big( r_{n} \cdot s_{n} + \frac{1}{n} \Big).
\end{align*}
This lower bound is attained by the first-order estimator $\hat{\psi}_{1, n}$. The bias, variance, and MSE of $\hat{\psi}_{1, n}$ have the following forms:
\begin{align*}
\bias (\hat{\psi}_{1, n}) & = - \langle a - \hat{a}, b - \hat{b} \rangle_{\bbP}, \\
\var (\hat{\psi}_{1, n}) & = \frac{1}{n} \left\{  \E [(A - \hat{a} (X))^{2} (Y - \hat{b} (X))^{2}] - \E^{2} [(A - \hat{a} (X)) (Y - \hat{b} (X))] \right\}, \text{ and} \\
\mse (\hat{\psi}_{1, n}) & = \langle a - \hat{a}, b - \hat{b} \rangle_{\bbP}^{2} + \frac{1}{n} \left\{ \E [(A - \hat{a} (X))^{2} (Y - \hat{b} (X))^{2}] - \E^{2} [(A - \hat{a} (X)) (Y - \hat{b} (X))] \right\}.
\end{align*}
\end{lemma}

To construct the second-order estimator, we similarly find a $k$-dimensional dictionary $\bar{\phi}_{k}$ and denote $\Sigma \coloneqq \E [\bar{\phi}_{k} (X)^{\otimes2}]$. In practice, one needs to estimate $\Sigma$ from data. We make the following assumptions on $\Sigma$ and its estimator.

\begin{assumption}
\label{as:Sigma} 
$\Sigma$ is assumed to have bounded spectra and there exists an estimator $\hat{\Sigma}$ of $\Sigma$ such that $\hat{\Sigma}$ also has bounded spectra and $\Vert \hat{\Sigma} - \Sigma\Vert_{\op} = o (1)$, where $\Vert \cdot \Vert_{\op}$ denotes the matrix operator norm. Without loss of generality, we take $\Sigma = \Sigma^{-1} = \mathrm{I}$.
\end{assumption}

We then construct the following second-order estimator for $\psi (a, b)$.
\begin{align*}
& \hat{\psi}_{2, n} (\bar{\phi}_{k}; \hat{\Sigma}) \coloneqq \hat{\psi}_{1, n} + \hat{U}_{n, 2} (\bar{\phi}_{k}; \hat{\Sigma}), \text{ where } \\
& \hat{U}_{n, 2} (\bar{\phi}_{k}; \hat{\Sigma}) \coloneqq \bbU_{n, 2} \left[ (A_{1} - \hat{a} (X_{1})) \bar{\phi}_{k} (X_{1})^{\top} \hat{\Sigma}^{-1} \bar{\phi}_{k} (X_{2}) (Y_{2} - \hat{b} (X_{2})) \right].
\end{align*}

We further introduce some short-hand notation for ease of exposition:
\begin{align*}
& \hat{\varepsilon}_{a} \coloneqq A - \hat{a} (X), \hat{\varepsilon}_{b} \coloneqq Y - \hat{b} (X),\\
& \alpha \coloneqq \E [(a (X) - \hat{a} (X)) \bar{\phi}_{k} (X)], \beta \coloneqq [(b (X) - \hat{b} (X)) \bar{\phi}_{k} (X)], \\
& \Sigma_{a, a} \coloneqq \E [(A - \hat{a} (X))^{2} \bar{\phi}_{k} (X) \bar{\phi}_{k} (X)^{\top}], \Sigma_{b, b} \coloneqq \E [(Y - \hat{b} (X))^{2} \bar{\phi}_{k} (X) \bar{\phi}_{k} (X)^{\top}], \\
& \text{and } \Sigma_{a, b} \coloneqq \E [(A - \hat{a} (X)) (Y - \hat{b} (X)) \bar{\phi}_{k} (X) \bar{\phi}_{k} (X)^{\top}].
\end{align*}

We are now ready to state the following lemma. The proof is by direct calculations and can be found in Appendix~\ref{app:lem:ecc2}.

\begin{lemma}
\label{lem:ecc2} 
The bias and variance of $\hat{\psi}_{2, n} (\bar{\phi}_{k}; \hat{\Sigma})$ read as:
\begin{align*}
\bias (\hat{\psi}_{2, n} (\bar{\phi}_{k}; \hat{\Sigma})) & = - \langle \mathsf{\Pi}^{\perp} [\hat{a} - a \mid \bar{\phi}_{k}], \mathsf{\Pi}^{\perp} [\hat{b} - b \mid \bar{\phi}_{k}] \rangle_{\bbP} + \alpha^{\top} (\hat{\Sigma}^{-1} - \mathrm{I}) \beta \lesssim r_{n}^{1 / 2} \cdot s_{n}^{1 / 2}, \\
\var (\hat{\psi}_{2, n} (\bar{\phi}_{k}; \hat{\Sigma})) & = \var (\hat{\psi}_{1, n}) + \var (\hat{U}_{n, 2} (\bar{\phi}_{k}; \hat{\Sigma})) + 2 \cov (\hat{\psi}_{1, n}, \hat{U}_{n, 2} (\bar{\phi}_{k}; \hat{\Sigma})) \lesssim \frac{1}{n},
\end{align*}
where
\begin{align*}
\var (\hat{\psi}_{1, n}) =  & \ \frac{1}{n} \left\{ \E [\hat{\varepsilon}_{a}^{2} \hat{\varepsilon}_{b}^{2}] - \E^{2} [\hat{\varepsilon}_{a} \hat{\varepsilon}_{b}] \right\}, \\
\var (\hat{U}_{n, 2} (\bar{\phi}_{k}; \hat{\Sigma})) = & \ \frac{1}{n (n - 1)} \Tr \left\{ \Sigma_{a, a} \hat{\Sigma}^{-1} \Sigma_{b, b} \hat{\Sigma}^{-1} + (\Sigma_{a, b} \hat{\Sigma}^{-1})^{2} \right\} \\
& + \frac{n - 2}{n (n - 1)} \left( \alpha^{\top} \hat{\Sigma}^{-1} \Sigma_{b, b} \hat{\Sigma}^{-1} \alpha + \beta^{\top} \hat{\Sigma}^{-1} \Sigma_{a, a} \hat{\Sigma}^{-1} \beta+ 2 \alpha^{\top} \hat{\Sigma}^{-1} \Sigma_{a, b} \hat{\Sigma}^{-1} \beta \right) \\
& - \frac{2 (2 n - 3)}{n (n - 1)} (\alpha^{\top} \hat{\Sigma}^{-1} \beta)^{2}, \\
\leq & \ \frac{C k}{n^{2}} + \frac{C}{n} \left\{ \Vert \mathsf{\Pi} [\hat{a} - a \mid \bar{\phi}_{k}] \Vert_{2, \bbP}^{2} + \Vert \mathsf{\Pi} [\hat{b} - b \mid \bar{\phi}_{k}] \Vert_{2, \bbP}^{2} \right\} \lesssim \frac{1}{n}, \\
2 \mathsf{cov} (\hat{\psi}_{1, n}, \hat{U}_{n, 2} (\bar{\phi}_{k}; \hat{\Sigma})) = & \ \frac{2}{n} \left\{ \E [\hat{\varepsilon}_{a} \hat{\varepsilon}_{b}^{2} \bar{\phi}_{k} (X)^{\top}] \hat{\Sigma}^{-1} \alpha + \E [\hat{\varepsilon}_{a}^{2} \hat{\varepsilon}_{b} \bar{\phi}_{k} (X)^{\top}] \hat{\Sigma}^{-1} \beta - 2 \E [\hat{\varepsilon}_{a} \hat{\varepsilon}_{b}] \alpha^{\top} \hat{\Sigma}^{-1} \beta \right\} \\
\leq & \ \frac{C}{n} \left\{ \Vert \mathsf{\Pi} [\hat{a} - a \mid \bar{\phi}_{k}] \Vert_{2, \bbP} + \Vert \mathsf{\Pi} [\hat{b} - b \mid \bar{\phi}_{k}] \Vert_{2, \bbP} \right\} \lesssim \frac{1}{n}.
\end{align*}
The inequalities hold for $\bbP_{a, b} \in\calP_{\SA} ((\hat{a}, \hat{b}), (r_{n}, s_{n}))$ under Assumption~\ref{as:Sigma}.
\end{lemma}

It is not difficult to also see that $\hat{\psi}_{2, n} (\bar{\phi}_{k}; \hat{\Sigma})$ corrects the bias of $\hat{\psi}_{1, n}$ by estimating a lower bound of $\bias (\hat{\psi}_{1, n}) \lesssim r_{n}^{1 / 2} \cdot s_{n}^{1 / 2}$.

\begin{remark}
\label{rem:Sigma_hat} 
In \citet{liu2020nearly} and \citet{liu2017semiparametric}, we have shown that when $\hat{\Sigma}$ is the sample Gram matrix estimator $\Vert \hat{\Sigma} - \mathrm{I} \Vert_{\op} = \sqrt{k \log k / n} = o (1)$ when $k = o (n / \log^{2} n)$ when the sample used to compute $\hat{\Sigma}$ is also of size $n$ \citep{tropp2015introduction}. When we know $\Sigma = \mathrm{I}$, $\bias (\hat{\psi}_{2, n} (\bar{\phi}_{k}; \mathrm{I}))$ is reduced to $- \langle \mathsf{\Pi}^{\perp} [\hat{a} - a \mid \bar{\phi}_{k}], \mathsf{\Pi}^{\perp} [\hat{b} - b \mid \bar{\phi}_{k}] \rangle_{\bbP} \lesssim r_{n}^{1 / 2} \cdot s_{n}^{1 / 2}$ because there is no extra bias due to estimating $\Sigma$. However, even if we estimate $\Sigma$ by $\hat{\Sigma}$, the extra bias incurred is of the form
\begin{align*}
\alpha^{\top} (\hat{\Sigma}^{-1} - \mathrm{I}) \beta \lesssim \Vert \mathsf{\Pi} [\hat{a} - a \mid \bar{\phi}_{k}] \Vert_{2, \bbP} \Vert\mathsf{\Pi} [\hat{b} - a \mid \bar{\phi}_{k}] \Vert_{2, \bbP} \Vert\hat{\Sigma}^{-1} - \mathrm{I} \Vert_{\op} = o (r_{n}^{1 / 2} \cdot s_{n}^{1 / 2}).
\end{align*}
Thus as long as we have a consistent estimator of $\Sigma$, the second-order estimator is still asymptotically minimax under the SA model.
\end{remark}

\begin{theorem}
\label{thm:ecc} 
Under Model $\calP_{\SA} ((\hat{a}, \hat{b}), (r_{n}, s_{n}))$, both $\hat{\psi}_{1, n}$ and $\hat{\psi}_{2, n} (\bar{\phi}_{k}; \hat{\Sigma})$ are asymptotically minimax, as long as $\Sigma$ and $\hat{\Sigma}$ satisfy Assumption~\ref{as:Sigma}. The same conclusions hold when we replace the SA model $\calP_{\SA} ((\hat{a}, \hat{b}), (r_{n}, s_{n}))$ with the assumption-lean model $\calP_{\AL} ((\hat{a}, \hat{b}))$.
\end{theorem}

\begin{proof}
The minimaxity of $\hat{\psi}_{2, n} (\bar{\phi}_{k}; \hat{\Sigma})$ can be concluded using the orders of its bias and variance shown in Lemma~\ref{lem:ecc2}.
\end{proof}

\begin{remark}
\label{rem:ate} 
Similar statements to those in Theorem~\ref{thm:ecc} hold for the average treatment effect and the average treatment effect on the treated. For instance, for the treatment specific mean, this can be seen by replacing the notation $a, b, \hat{\varepsilon}_{a}, \hat{\varepsilon}_{b}$ by the following instead:
\begin{align*}
& a (\cdot) = 1 / \E [A \mid X = \cdot], b (\cdot) = \E [Y \mid X = \cdot, A = 1], \hat{\varepsilon}_{a} = A \hat{a} (X) - 1, \hat{\varepsilon}_{b} = A (Y - \hat{b} (X)).
\end{align*}
The dictionary $\bar{\phi}_{k}$ will also be weighted by the treatment indicator $A \bar{\phi}_{k}$. The minimaxity of the first-order DML estimators of these two functionals has been shown in \citet{jin2025structure}.
\end{remark}

\begin{remark}
\label{rem:dr} 
As indicated after Definition~\ref{def:al}, we will discuss in Section~\ref{sec:sa} that the higher-order generalization of the second-order estimators can be used to falsify the null hypothesis $\calH_{0}: \bias (\hat{\psi}_{1, n}) \ll n^{-1/2}$ when $\psi$ belongs to the \emph{monotone bias class}. If $\psi$ is the expected conditional covariance or the treatment specific mean parameter mentioned in Remark~\ref{rem:ate}, $\psi$ belongs to the so-called mixed-bias class \citep{rotnitzky2021characterization} but not the \emph{monotone bias class}. Here, we cannot claim the asymptotic inadmissibility of the DML estimator $\hat{\psi}_{1, n}$ of $\psi$ and similarly we cannot directly falsify $\calH_{0}: \bias (\hat{\psi}_{1, n}) \ll n^{-1 / 2}$. Nonetheless, we can empirically falsify the rate-double-robustness of $\hat{\psi}_{1, n}$ \citep{liu2024assumption}, where rate-double-robustness refers to the assumption $r_{n}^{1 / 2} \cdot s_{n}^{1 / 2} = o(n^{-1/2})$ for both the expected conditional covariance or the treatment specific mean parameter.
\end{remark}

\subsection*{Specializing to the Expected Conditional Variance}

A special case of the ECC functional--the expected conditional variance (abbreviated as the ECV functional) $\psi (a) \equiv \psi (a, a)$, however, belongs to the \emph{monotone bias class}, when $A = Y$ with probability 1. Here, the corresponding DML estimator is $\hat{\psi}_{1, n} \coloneqq n^{-1} \sum_{i = 1}^{n} (A_{i} - \hat{a} (X_{i}))^{2}$ and the corresponding SA model is defined as
\begin{align*}
\calP_{\SA} (\hat{a}, r_{n}) \equiv \calP_{\SA} ((\hat{a}, \hat{a}), (r_{n}, r_{n})) \coloneqq \Big\{ \bbP_{a}: \Vert \hat{a} - a \Vert_{2, \bbP}^{2} \leq r_{n} \Big\}.
\end{align*}

\begin{remark}
\label{rem:cf2} 
Similar to Remark~\ref{rem:cf1}, we use the same estimator $\hat{a} \equiv\hat{a}_{1} \equiv\hat{a}_{2}$ to compute $\hat{\psi}_{1, n}$, again excluding the estimators studied in \citet{newey2018cross, mcgrath2026nuisance, mcclean2026double}.
\end{remark}

Analogously, the second-order estimator for $\psi(a)$ takes the following form:
\begin{align*}
& \hat{\psi}_{2, n} (\bar{\phi}_{k}; \hat{\Sigma}) \coloneqq \hat{\psi}_{1, n} + \hat{U}_{n, 2} (\bar{\phi}_{k}; \hat{\Sigma}), \text{ where } \\
& \hat{U}_{n, 2} (\bar{\phi}_{k}; \hat{\Sigma}) \coloneqq \bbU_{n, 2} [(A_{1} - \hat{a} (X_{1})) \bar{\phi}_{k} (X_{1})^{\top} \hat{\Sigma}^{-1} \bar{\phi}_{k} (X_{2}) (A_{2} - \hat{a} (X_{2}))].
\end{align*}
Lemma~\ref{lem:ecc} and Lemma~\ref{lem:ecc2} immediately imply the two corollaries below for the ECV functional $\psi(a)$.

\begin{corollary}
\label{cor:ecc} 
The following hold. There exists a universal constant $C > 0$ such that
\begin{align*}
\mathfrak{R}_{n} (\psi; \calP_{\SA} (\hat{a}, r_{n})) := \inf_{\hat{\psi}} \sup_{\bbP_{a} \in\calP_{\SA} (\hat{a}, r_{n})} \E_{a} [(\hat{\psi} - \psi(a))^{2}] \geq C \Big( r_{n}^{2} + \frac{1}{n} \Big).
\end{align*}
This lower bound is attained by the first-order estimator $\hat{\psi}_{1, n}$. The bias, variance, and MSE of $\hat{\psi}_{1, n}$ have the following forms:
\begin{align*}
\bias (\hat{\psi}_{1, n}) & = - \Vert a - \hat{a} \Vert_{2, \bbP}^{2} = - \, \Vert\mathsf{\Pi}^{\perp} [\hat{a} - a \mid \bar{\phi}_{k}] \Vert_{2, \bbP}^{2} - \alpha^{\top} \alpha, \\
\var (\hat{\psi}_{1, n}) & = \frac{1}{n} \left\{ \E [\hat{\varepsilon}_{a}^{4}] - \E^{2} [\hat{\varepsilon}_{a}^{2}] \right\}, \text{ and} \\
\mse (\hat{\psi}_{1, n}) & = \Vert a - \hat{a} \Vert_{2, \bbP}^{4} + \frac{1}{n} \left\{  \E [\hat{\varepsilon}_{a}^{4}] - \E^{2} [\hat{\varepsilon}_{a}^{2}] \right\}.
\end{align*}
\end{corollary}

\begin{corollary}
\label{cor:ecv2} 
The bias and variance of $\hat{\psi}_{2, n} (\bar{\phi}_{k}; \hat{\Sigma})$ read as:
\begin{align*}
\bias (\hat{\psi}_{2, n} (\bar{\phi}_{k}; \hat{\Sigma})) & = - \Vert\mathsf{\Pi}^{\perp} [\hat{a} - a \mid \bar{\phi}_{k}] \Vert_{2, \bbP}^{2} + \alpha^{\top} (\hat{\Sigma}^{-1} - \mathrm{I}) \alpha\lesssim r_{n},\\
\var (\hat{\psi}_{2, n} (\bar{\phi}_{k}; \hat{\Sigma})) & = \var (\hat{\psi}_{1, n}) + \var (\hat{U}_{n, 2} (\bar{\phi}_{k}; \hat{\Sigma})) + 2 \mathsf{cov} (\hat{\psi}_{1, n}, \hat{U}_{n, 2} (\bar{\phi}_{k}; \hat{\Sigma})) \lesssim\frac{1}{n},
\end{align*}
where
\begin{align*}
\var (\hat{\psi}_{1, n}) =  &  \ \frac{1}{n} \var (\hat{\varepsilon}_{a}^{2}), \\
\var (\hat{U}_{n, 2} (\bar{\phi}_{k}; \hat{\Sigma})) = & \ \frac{2}{n (n - 1)} \mathsf{Tr} \left\{  (\Sigma_{a, a} \hat{\Sigma}^{-1})^{2} \right\}  + \frac{4 n - 8}{n (n - 1)} \left(  \alpha^{\top} \hat{\Sigma}^{-1} \Sigma_{a, a} \hat{\Sigma}^{-1} \alpha\right) \\
& - \frac{4 n - 6}{n (n - 1)} (\alpha^{\top} \hat{\Sigma}^{-1} \alpha)^{2} \\
\leq & \ \frac{C k}{n^{2}} + \frac{C}{n} \Vert\mathsf{\Pi}[\hat{a} - a \mid \bar{\phi}_{k}] \Vert_{2, \bbP}^{2} \lesssim\frac{1}{n}, \\
2 \mathsf{cov} (\hat{\psi}_{1, n}, \hat{U}_{n, 2} (\bar{\phi}_{k}; \hat{\Sigma})) = & \ \frac{4}{n} \left\{  \E [\hat{\varepsilon}_{a}^{3} \bar{\phi}_{k} (X)^{\top}] \hat{\Sigma}^{-1} \alpha- \E [\hat{\varepsilon}_{a}^{2}] \alpha^{\top} \hat{\Sigma}^{-1} \alpha\right\} \\
\leq & \ \frac{C}{n} \Vert\mathsf{\Pi}[\hat{a} - a \mid \bar{\phi}_{k}] \Vert_{2, \bbP} \lesssim\frac{1}{n}.
\end{align*}
The inequalities hold for $\bbP_{a} \in\calP_{\SA} (\hat{a}, r_{n})$ under Assumption~\ref{as:Sigma}.
\end{corollary}

By Corollary~\ref{cor:ecv2}, $\psi (a)$ belongs to the \emph{monotone bias class}. By piecing together the above two corollaries, we obtain the final theoretical result of this paper. There always exists a distribution in $\calP_{\SA} (\hat{a}, r_{n})$ such that Definition~\ref{def:mbc}(1) holds. To see this, consider the case where $a = \hat{a} + r_{n}^{1 / 2} \beta^{\top} \bar{\phi}_{k}$ with $\Vert \beta \Vert_{2} = 1$, for which $\bias (\hat{\psi}_{2, n} (\bar{\phi}_{k}; \hat{\Sigma})) = \alpha^{\top} (\hat{\Sigma}^{-1} - \mathrm{I}) \alpha \ll \bias (\hat{\psi}_{1, n}) = \alpha^{\top} \alpha$. The rest of the proof can be found in Appendix~\ref{app:ecv}.

\begin{theorem}
\label{thm:ecv} 
Choosing $k$ as in Remark~\ref{rem:Sigma_hat} such that $k = o (n)$ and $\Vert \hat{\Sigma}^{-1} - \Sigma^{-1} \Vert_{\op} = o (1)$. Under Model $\calP_{\SA} (\hat{a}, r_{n})$ and Assumption~\ref{as:Sigma}, $\hat{\psi}_{2, n} (\bar{\phi}_{k}; \hat{\Sigma})$ is \emph{asymptotically minimax}. The same conclusions as in Theorem~\ref{thm:high level} hold by replacing $\psi (\theta)$ as $\psi (a)$, $\hat{\psi}_{1, n}$ as $\hat{\psi}_{1, n}$, and $\hat{\psi}_{2, n}$ as $\hat{\psi}_{2, n} (\bar{\phi}_{k}; \hat{\Sigma})$. The same conclusions also hold when we replace the SA model $\calP_{\SA} (\hat{a}, r_{n})$ with the assumption-lean model $\calP_{\AL} (\hat{a})$ and drop the assumption on $r_{n}$.
\end{theorem}

\begin{remark}
\label{rem:Sigma_hat var} 
We note that $\bias (\hat{\psi}_{2, n} (\bar{\phi}_{k}; \hat{\Sigma}))^{2} - \bias (\hat{\psi}_{1, n})^{2} \asymp- \Vert\mathsf{\Pi}[\hat{a} - a \mid \bar{\phi}_{k}] \Vert_{2, \bbP}^{2} (1 - \Vert\hat{\Sigma}^{-1} - \mathrm{I} \Vert_{\mathrm{op}}) \asymp - \Vert \mathsf{\Pi} [\hat{a} - a \mid \bar{\phi}_{k}] \Vert_{2, \bbP}^{2}$ by Assumption~\ref{as:Sigma}. Thus, asymptotically, the second-order estimator still has smaller bias than the first-order DML estimator $\hat{\psi}_{1, n}$, and the impact of estimating $\Sigma$ is asymptotically negligible. In addition, $\hat{\psi}_{2, n} (\bar{\phi}_{k}; \hat{\Sigma})$ corrects the bias of $\hat{\psi}_{1, n}$ by estimating a lower bound of $\bias (\hat{\psi}_{1, n}) \lesssim r_{n}$.
\end{remark}

\section{Concluding Remarks}
\label{sec:sa}

The SA model introduced in \citet{balakrishnan2026fundamental} is a mathematically appealing construct that has inspired follow-up work \citep{jin2025structure, bonvini2024doubly, jin2025normal, jin2025sharp, gu2026optimally}, including our current paper. The assumption-lean model $\calP_{\AL}(\hat{\theta})$ defined in Definition~\ref{def:al}, is aligned with the goal of understanding what can be learned from a model that makes almost no assumptions. As discussed earlier, in terms of point estimation, both the first-order DML estimator $\hat{\psi}_{1, n}$ and our second-order estimator $\hat{\psi}_{2, n}$ remain minimax with rate $O (1)$ in the assumption-lean model $\calP_{\AL}(\hat{\theta})$ for all the parameters studied in \citet{balakrishnan2026fundamental}; for $\psi$ in \emph{monotone bias class}, our $\hat{\psi}_{2, n}$ continues to asymptotically dominate $\hat{\psi}_{1, n}$ in the scaled MSE.

However, statisticians care about uncertainty quantification or inference as much as or even more than point estimation. Neither the minimaxity and the asymptotic inadmissibility of $\hat{\psi}_{1, n}$ nor the minimaxity of $\hat{\psi}_{2, n}$ in the assumption-lean model $\calP_{\AL}(\hat{\theta})$ offer any guidance on how to quantify uncertainty, absent further knowledge of $\hat{\theta}$ or $\Theta$. The above argument is not new. Before \citet{balakrishnan2026fundamental}, we considered inference on $\psi$ under the assumption-lean model $\calP_{\AL} (\hat{\theta})$ in \citet{liu2020nearly} and \citet{liu2024assumption}. The former paper was discussed by the authors of \citet{balakrishnan2026fundamental}; see \citet{kennedy2020discussion} and \citet{liu2020rejoinder}. Since no uniformly consistent estimators of $\psi$ exist in model $\calP_{\AL}(\hat{\theta})$ \citep{ritov1990achieving, robins1997toward, ritov2014bayesian}, we, instead, developed valid falsification tests of the following null hypothesis for $\psi$ in the \emph{monotone bias class}:

\begin{quote}
$\calH_{0}$: \emph{The bias of the first-order DML estimator $\hat{\psi}_{1, n}$ of $\psi$ is smaller than $\delta \cdot \var^{1 / 2} (\hat{\psi}_{1, n})$ for some small $\delta > 0$, so that a standard Wald CI centered at $\hat{\psi}_{1, n}$ has an asymptotic coverage no smaller than some lower bound of the nominal coverage determined by $\delta$.}
\end{quote}

The proposed tests are only falsification tests because, although valid under $\mathcal{H}_{0}$, they will have no power under many alternatives to $\mathcal{H}_{0}$. However, when a test rejects the null, it provides empirical evidence that the bias of $\hat{\psi}_{1,n}$ is too large for the Wald CI to deliver valid inference. The test statistics used are based on the same second-order estimators that we have analyzed in this paper or their higher-order extensions \citep{robins2008higher, robins2016technical, liu2017semiparametric}. For $\psi$ belonging to the so-called mixed-bias classes (which includes the expected conditional covariance analyzed above) \citep{rotnitzky2021characterization}, in \citet{liu2020nearly} and \citet{liu2024assumption}, we showed that these tests are no longer valid under $\mathcal{H}_{0}$. However, these tests remain valid falsification tests of the so-called rate-double-robustness property, as defined in Remark~\ref{rem:high} or Remark~\ref{rem:dr}. We note that the rate-double-robustness implies that $\mathcal{H}_{0}$ is true. For this reason, complexity-reducing assumptions strong enough to imply rate-double-robustness are often made by investigators to justify the validity of their Wald CIs centering $\hat{\psi}_{1,n}$. In our view, unlike the minimaxity of $\hat{\psi}_{1,n}$ or of $\hat{\psi}_{2,n}$, these falsification tests provide further empirical information even in the assumption-lean model $\calP_{\AL}(\hat{\theta})$, whenever they reject and thus can be of value to domain scientists for whom inference is important. Finally, we remark that the uniform asymptotic inadmissibility of $\hat{\psi}_{1, n}$ in the \emph{monotone bias class}, in fact, follows from a similar intuition to that of the above falsification test procedure, a point implicitly made in Remark~1.7 of \citet{liu2020nearly}. At any finite sample size $n$, we allow $\hat{\psi}_{2, n}$ to be $\eps$-worse than $\hat{\psi}_{1, n}$ in the most pessimistic scenario, for some very small $\eps$ that depends on $n$; but as a reward, $\hat{\psi}_{2, n}$ could potentially be much better than $\hat{\psi}_{1, n}$ in certain optimistic cases, either in the structure-agnostic or assumption-lean model.

\bibliographystyle{plainnat}
\bibliography{\myreferences}

\appendix

\section{Proofs of the Main Theorems}
\label{app:proofs}

\subsection{Proof of Theorem~\ref{thm:quad gsm}}
\label{app:quad gsm}

\begin{proof}
The first statement on the minimaxity of $\hat{Q}_{2, n} (k)$ follows immediately from its bias and variance characterized in Lemma~\ref{lem:quad gsm 2}.

The second statement holds by noticing that Lemma~\ref{lem:quad gsm 2} implies that $Q (\theta)$ satisfies the condition in Definition~\ref{def:mbc}, and hence Theorem~\ref{thm:high level} applies.

The last statement has been proved in the text before Theorem~\ref{thm:quad gsm}.
\end{proof}

\subsection{Proof of Theorem~\ref{thm:quad}}
\label{app:quad}

\begin{proof}
The first statement on the minimaxity of $\hat{\psi}_{2, n} (\bar{\phi}_{k})$ follows immediately from its bias and variance characterized in Lemma~\ref{lem:quad2}.

The second statement holds by noticing that Lemma~\ref{lem:quad gsm 2} implies that $\psi (f)$ satisfies the condition in Definition~\ref{def:mbc}, and hence Theorem~\ref{thm:high level} applies.

The last statement has been proved in the text before Theorem~\ref{thm:quad}.
\end{proof}

\subsection{Proof of Theorem~\ref{thm:ecv}}
\label{app:ecv}

\begin{proof}
The first statement on the minimaxity of $\hat{\psi}_{2, n} (\bar{\phi}_{k}; \hat{\Sigma})$ follows immediately from its bias and variance characterized in Corollary~\ref{cor:ecv2}.

The second statement holds by noticing that Lemma~\ref{lem:quad gsm 2} implies that $\psi (a)$ satisfies the condition in Definition~\ref{def:mbc}, and hence Theorem~\ref{thm:high level} applies.

The last statement has been proved in the text before Theorem~\ref{thm:ecv}.
\end{proof}

\section{Discussion on the Scaled MSE Difference}
\label{app:scale}

In this section, we discuss why we choose the scaled MSE difference as the criterion in Definition~\ref{def:admissible} and discuss the consequence of using other scaling factors instead of $\mse (\hat{\psi}_{1, n})$ in all the examples analyzed in the main text. To simplify the presentation, we use $\hat{\psi}_{2, n}$ to denote a generic second-order estimator. Scaling the MSE difference by the minimax rate of $\hat{\psi}_{1, n}$ is similar to scaling the difference by the MSE of $\hat{\psi}_{1, n}$. Below we consider several other natural scaling choices.

First, we consider the unscaled MSE difference: $\mse (\hat{\psi}_{2, n}) - \mse (\hat{\psi}_{1, n})$. Without any scaling, it is not difficult to see that for all the examples in the \emph{monotone bias class} (Definition~\ref{def:mbc}), we have that as $n \to \infty$, $\mse (\hat{\psi}_{2, n}) - \mse (\hat{\psi}_{1, n})$ will always be non-positive, depending on the difference in biases. We find such a comparison overly simplifying the difference between $\hat{\psi}_{2, n}$ and $\hat{\psi}_{1, n}$ because it does not take into account the possible variance inflation when using $\hat{\psi}_{2, n}$, a disadvantage of $\hat{\psi}_{2, n}$ in particular in finite samples.

Second, we consider the MSE difference scaled by $\var (\hat{\psi}_{1, n}) \equiv v / n$. Then by the same argument as in the proof of Theorem~\ref{thm:high level}, as long as $n \geq N$, $\frac{\mse (\hat{\psi}_{2, n}) - \mse (\hat{\psi}_{1, n})}{\var (\hat{\psi}_{1, n})} \leq \eps$.

Finally, we consider the difference in MSEs, scaled only by $\bias (\hat{\psi}_{1, n})^{2}$ but not the variance. In this case, when $\bias (\hat{\psi}_{1, n})^{2} = o (1 / n)$, since $\var (\hat{\psi}_{2, n}) - \var (\hat{\psi}_{1, n})$ often has a term of order $k / n^{2} = o (1 / n)$ by the choice of $k = o (n)$, then the ratio between the variance difference and $\bias (\hat{\psi}_{1, n})^{2}$ may diverge to $+ \infty$. However, we want to emphasize that this will only happen when the first-order DML estimator $\hat{\psi}_{1, n}$ has bias faster than parametric rates. If this were true, bias alone cannot reflect the statistical behavior of $\hat{\psi}_{1, n}$ due to the $O (1 / n)$ variance, and thus this scaling choice is not very meaningful statistically.

\section{Proofs of Lemmas in the Main Text}
\label{app:details}

In this part, we record the proofs of several lemmas in the main paper. The proofs mainly involve elementary calculations related to $U$-statistics.

\subsection{Proof of Lemma~\ref{lem:quad gsm 2}}
\label{app:lem:quad gsm 2}

\begin{proof}
The bias formula of $\hat{Q}_{2, n} (k)$ is trivial. The variance formula is derived by combining the following three results. Let $Z \sim N (0, n^{-1})$. First,
\begin{align*}
\var (\hat{Q}_{1, n}) = 4 \var (\langle Y, \hat{\theta} \rangle) = \frac{4}{n} \Vert \hat{\theta} \Vert_{2}^{2}.
\end{align*}
Next,
\begin{align*}
& \ \var (\hat{Q}_{2, n} (k) - \hat{Q}_{1, n}) \\
= & \ \sum_{j = 1}^{k} \var (Y_{j, 1} Y_{j, 2} - 2 Y_{j} \hat{\theta}_{j}) \\
= & \ \sum_{j = 1}^{k} \left( \sfE^{2} [Y_{j, 1}^{2}] - \sfE^{4} [Y_{j, 1}] + 4 \hat{\theta}_{j}^{2} \var (Y_{j}) - 4 \hat{\theta}_{j} \cov (Y_{j, 1} Y_{j, 2}, Y_{j}) \right) \\
= & \ \sum_{j = 1}^{k} \left( \sfE^{2} [(Y_{j} + Z)^{2}] - \sfE^{4} [Y_{j} + Z] + 4 \hat{\theta}_{j}^{2} \frac{1}{n} - 4 \hat{\theta}_{j} \{\sfE [(Y_{j} + Z) (Y_{j} - Z) Y_{j}] - \sfE^{2} [Y_{j} + Z] \sfE [Y_{j}]\} \right) \\
= & \ \sum_{j = 1}^{k} \left\{ \left( \frac{2}{n} + \theta_{j}^{2} \right)^{2} - \theta_{j}^{4} + \frac{4}{n} \hat{\theta}_{j}^{2} - 4 \hat{\theta}_{j} \left( \sfE [Y_{j}^{3}] - \sfE [Z^{2}] \theta_{j} - \theta_{j}^{3} \right) \right\} \\
= & \ \sum_{j = 1}^{k} \left\{ \frac{4}{n^{2}} + \frac{4}{n} \theta_{j}^{2} + \frac{4}{n} \hat{\theta}_{j}^{2} - 4 \hat{\theta}_{j} \left( \frac{3}{n} \theta_{j} - \frac{1}{n} \theta_{j} \right) \right\} \\
= & \ \frac{4 k}{n^{2}} + \frac{4}{n} \sum_{j = 1}^{k} (\hat{\theta}_{j} - \theta_{j})^{2} \equiv \frac{4 k}{n^{2}} + \frac{4}{n} \Vert \sfPi_{k} (\hat{\theta} - \theta) \Vert_{2}^{2}.
\end{align*}
Finally,
\begin{align*}
& \ 2 \cov (\hat{Q}_{1, n}, \hat{Q}_{2, n} (k) - \hat{Q}_{1, n}) \\
= & \ 2 \sum_{j = 1}^{k} \cov \left( \langle Y, \hat{\theta} \rangle, Y_{j, 1} Y_{j, 2} - 2 Y_{j} \hat{\theta}_{j} \right) \\
= & \ 2 \sum_{j = 1}^{k} \cov \left( Y_{j} \hat{\theta}_{j}, Y_{j}^{2} - Z^{2} - 2 Y_{j} \hat{\theta}_{j} \right) \\
= & \ 2 \sum_{j = 1}^{k} \hat{\theta}_{j} \left\{ \theta_{j}^{3} + \frac{3}{n} \theta_{j} - \theta_{j} \left( \theta_{j}^{2} + \frac{1}{n} \right) - \frac{2}{n} \hat{\theta}_{j} \right\} \\
= & \ \frac{4}{n} \sum_{j = 1}^{k} \hat{\theta}_{j} (\theta_{j} - \hat{\theta}_{j}) = \frac{4}{n} \langle \sfPi_{k} \theta, \sfPi_{k} \hat{\theta} \rangle - \frac{4}{n} \Vert \sfPi_{k} \hat{\theta} \Vert_{2}^{2}.
\end{align*}
\end{proof}

\subsection{Proof of Lemma~\ref{lem:quad2}}
\label{app:lem:quad2}

\begin{proof}
The bias formula of $\hat{\psi}_{2, n} (\bar{\phi}_{k})$ is trivial. Recall that $\eta = \int \bar{\phi}_{k} (x) f (x) \diff x$ and $\hat{\eta} = \int \bar{\phi}_{k} (x) \hat{f} (x) \diff x$. The variance formula is derived by combining the following three results. First,
\begin{align*}
\var (\hat{\psi}_{1, n}) = \frac{4}{n} \var [\hat{f} (X)].
\end{align*}
Next, we have
\begin{align*}
& \ \var \Big( \bbU_{n, 2} \Big[ \Big( \bar{\phi}_{k} (X_{1}) - \int \bar{\phi}_{k} (x) \hat{f} (x) \diff x \Big)^{\top} \Big( \bar{\phi}_{k} (X_{2}) - \int \bar{\phi}_{k} (x) \hat{f} (x) \diff x \Big) \Big] \Big) \\
= & \ \frac{2}{n (n - 1)} \Big\{ \Tr (\Sigma^{2}) - 4 \hat{\eta}^{\top} \Sigma \eta + 2 \hat{\eta}^{\top} \Sigma \hat{\eta} + 2 \hat{\eta}^{\top} \hat{\eta} \cdot \eta^{\top} \eta + 2 (\hat{\eta}^{\top} \eta)^{2} - 4 \hat{\eta}^{\top} \hat{\eta} \cdot \hat{\eta}^{\top} \eta + (\hat{\eta}^{\top} \hat{\eta})^{2} \Big\} \\
& + \frac{4 n - 8}{n (n - 1)} (\hat{\eta} - \eta)^{\top} \Big( \Sigma - \hat{\eta} \eta^{\top} - \eta \hat{\eta}^{\top} + \hat{\eta} \hat{\eta}^{\top} \Big) (\hat{\eta} - \eta).
\end{align*}
Finally, we can compute the covariance term:
\begin{align*}
& \ 2 \cov \Big( \bbU_{n, 1} [2 \hat{f} (X)], \bbU_{n, 2} \Big[ \Big( \bar{\phi}_{k} (X_{1}) - \int \bar{\phi}_{k} (x) \hat{f} (x) \diff x \Big)^{\top} \Big( \bar{\phi}_{k} (X_{2}) - \int \bar{\phi}_{k} (x) \hat{f} (x) \diff x \Big) \Big] \Big) \\
= & \ \frac{8}{n} \Big( \int \bar{\phi}_{k} (x) f (x) \hat{f} (x) \diff x - \hat{\eta} \Big)^{\top} (\eta - \hat{\eta}).
\end{align*}
\end{proof}

\subsection{Proof of Lemma~\ref{lem:ecc2}}
\label{app:lem:ecc2}

In the proof, we only consider the case where $\Sigma$ is replaced by a generic estimator $\hat{\Sigma}$ treated as fixed, just as $\hat{a}$ and $\hat{b}$. To facilitate the proof, we let $\hat{\psi}_{2, n} (\bar{\phi}_{k}; \hat{\Sigma}) \equiv \hat{\psi}_{1, n} + \hat{U}_{2, n} (\bar{\phi}_{k}; \hat{\Sigma})$.

\begin{proof}
The bias of $\hat{\psi}_{2, n} (\bar{\phi}_{k}; \hat{\Sigma})$ is relatively simple to compute.
\begin{align*}
& \bias (\hat{\psi}_{2, n} (\bar{\phi}_{k}; \hat{\Sigma})) \\
& = \bias (\hat{\psi}_{1, n}) + \sfE [\sfPi (a - \hat{a} \mid \bar{\phi}_{k}) (X) \sfPi (b - \hat{b} \mid \bar{\phi}_{k}) (X)] + \alpha^{\top} (\hat{\Sigma}^{-1} - \I) \beta \\
& = - \sfE [\sfPi^{\perp} (a - \hat{a} \mid \bar{\phi}_{k}) (X) \sfPi^{\perp} (b - \hat{b} \mid \bar{\phi}_{k}) (X)] + \alpha^{\top} (\hat{\Sigma}^{-1} - \I) \beta \\
& \equiv - \langle \sfPi^{\perp} (a - \hat{a} \mid \bar{\phi}_{k}), \sfPi^{\perp} (b - \hat{b} \mid \bar{\phi}_{k}) \rangle_{\bbP} + \alpha^{\top} (\hat{\Sigma}^{-1} - \I) \beta.
\end{align*}
For variance, we directly compute each term in the variance decomposition: $\var \{\hat{\psi}_{2, n} (\bar{\phi}_{k})\} \equiv \var (\hat{\psi}_{1, n}) + \var \{\hat{U}_{n, 2} (\bar{\phi}_{k})\} + 2 \cov \{\hat{\psi}_{1, n}, \hat{U}_{n, 2} (\bar{\phi}_{k})\}$. We first compute the variance of $\hat{\psi}_{1, n}$.
\begin{align*}
\var (\hat{\psi}_{1, n}) & = \frac{1}{n} \var \{(A - \hat{a} (X)) (Y - \hat{b} (X))\} \\
& = \frac{1}{n} [\sfE \{(A - \hat{a} (X))^{2} (Y - \hat{b} (X))^{2}\} - \sfE^{2} \{(A - \hat{a} (X)) (Y - \hat{b} (X))\}].
\end{align*}
We then compute the variance of the bias correction term.
\begin{align*}
& \ \var \{\hat{U}_{n, 2} (\bar{\phi}_{k}; \hat{\Sigma})\} \\
= & \ \var \left( \bbU_{n, 2} [(A_{1} - \hat{a} (X_{1})) \bar{\phi}_{k} (X_{1})^{\top} \hat{\Sigma}^{-1} \bar{\phi}_{k} (X_{2}) (Y_{2} - \hat{b} (X_{2})] \right) \\
= & \ \frac{1}{n^{2} (n - 1)^{2}} \sum_{1 \leq i_{1} \neq i_{2} \leq n} \var [(A_{i_{1}} - \hat{a} (X_{i_{1}})) \bar{\phi}_{k} (X_{i_{1}})^{\top} \hat{\Sigma}^{-1} \bar{\phi}_{k} (X_{i_{2}}) (Y_{i_{2}} - \hat{b} (X_{i_{2}}))] \\
& + \frac{1}{n^{2} (n - 1)^{2}} \sum_{1 \leq i_{1} \neq i_{2} \leq n} \cov \left[ \begin{array}{c}
(A_{i_{1}} - \hat{a} (X_{i_{1}})) \bar{\phi}_{k} (X_{i_{1}})^{\top} \hat{\Sigma}^{-1} \bar{\phi}_{k} (X_{i_{2}}) (Y_{i_{2}} - \hat{b} (X_{i_{2}})), \\
(Y_{i_{1}} - \hat{b} (X_{i_{1}})) \bar{\phi}_{k} (X_{i_{1}})^{\top} \hat{\Sigma}^{-1} \bar{\phi}_{k} (X_{i_{2}}) (A_{i_{2}} - \hat{a} (X_{i_{2}}))
\end{array} \right] \\
& + \frac{1}{n^{2} (n - 1)^{2}} \sum_{1 \leq i_{1} \neq i_{2} \neq i_{3} \leq n} \cov \left[ \begin{array}{c}
(A_{i_{1}} - \hat{a} (X_{i_{1}})) \bar{\phi}_{k} (X_{i_{1}})^{\top} \hat{\Sigma}^{-1} \bar{\phi}_{k} (X_{i_{2}}) (Y_{i_{2}} - \hat{b} (X_{i_{2}})), \\
(A_{i_{1}} - \hat{a} (X_{i_{1}})) \bar{\phi}_{k} (X_{i_{1}})^{\top} \hat{\Sigma}^{-1} \bar{\phi}_{k} (X_{i_{3}}) (Y_{i_{3}} - \hat{b} (X_{i_{3}}))
\end{array} \right] \\
& + \frac{1}{n^{2} (n - 1)^{2}} \sum_{1 \leq i_{1} \neq i_{2} \neq i_{3} \leq n} \cov \left[ \begin{array}{c}
(A_{i_{1}} - \hat{a} (X_{i_{1}})) \bar{\phi}_{k} (X_{i_{1}})^{\top} \hat{\Sigma}^{-1} \bar{\phi}_{k} (X_{i_{2}}) (Y_{i_{2}} - \hat{b} (X_{i_{2}})), \\
(A_{i_{3}} - \hat{a} (X_{i_{3}})) \bar{\phi}_{k} (X_{i_{3}})^{\top} \hat{\Sigma}^{-1} \bar{\phi}_{k} (X_{i_{2}}) (Y_{i_{2}} - \hat{b} (X_{i_{2}}))
\end{array} \right] \\
& + \frac{2}{n^{2} (n - 1)^{2}} \sum_{1 \leq i_{1} \neq i_{2} \neq i_{3} \leq n} \cov \left[ \begin{array}{c}
(A_{i_{1}} - \hat{a} (X_{i_{1}})) \bar{\phi}_{k} (X_{i_{1}})^{\top} \hat{\Sigma}^{-1} \bar{\phi}_{k} (X_{i_{2}}) (Y_{i_{2}} - \hat{b} (X_{i_{2}})), \\
(A_{i_{2}} - \hat{a} (X_{i_{2}})) \bar{\phi}_{k} (X_{i_{2}})^{\top} \hat{\Sigma}^{-1} \bar{\phi}_{k} (X_{i_{3}}) (Y_{i_{3}} - \hat{b} (X_{i_{3}}))
\end{array} \right] \\
= & \ \frac{1}{n (n - 1)} \sfE \left[ (A_{1} - \hat{a} (X_{1}))^{2} (Y_{2} - \hat{b} (X_{2}))^{2} \bar{\phi}_{k} (X_{1})^{\top} \hat{\Sigma}^{-1} \bar{\phi}_{k} (X_{2}) \bar{\phi}_{k} (X_{2})^{\top} \hat{\Sigma}^{-1} \bar{\phi}_{k} (X_{1}) \right] \\
& + \frac{1}{n (n - 1)} \sfE \left[ (A_{1} - \hat{a} (X_{1})) (Y_{1} - \hat{b} (X_{1})) (A_{2} - \hat{a} (X_{2})) (Y_{2} - \hat{b} (X_{2})) \bar{\phi}_{k} (X_{1})^{\top} \hat{\Sigma}^{-1} \bar{\phi}_{k} (X_{2}) \bar{\phi}_{k} (X_{2})^{\top} \hat{\Sigma}^{-1} \bar{\phi}_{k} (X_{1}) \right] \\
& + \frac{n - 2}{n (n - 1)} \sfE \left[ (A_{1} - \hat{a} (X_{1}))^{2} (Y_{2} - \hat{b} (X_{2})) (Y_{3} - \hat{b} (X_{3})) \bar{\phi}_{k} (X_{2})^{\top} \hat{\Sigma}^{-1} \bar{\phi}_{k} (X_{1}) \bar{\phi}_{k} (X_{1})^{\top} \hat{\Sigma}^{-1} \bar{\phi}_{k} (X_{3}) \right] \\
& + \frac{n - 2}{n (n - 1)} \sfE \left[ (Y_{1} - \hat{b} (X_{1}))^{2} (A_{2} - \hat{a} (X_{2})) (A_{3} - \hat{a} (X_{3})) \bar{\phi}_{k} (X_{2})^{\top} \hat{\Sigma}^{-1} \bar{\phi}_{k} (X_{1}) \bar{\phi}_{k} (X_{1})^{\top} \hat{\Sigma}^{-1} \bar{\phi}_{k} (X_{3}) \right] \\
& + \frac{2 (n - 2)}{n (n - 1)} \sfE \left[ (Y_{1} - \hat{b} (X_{1})) (A_{1} - \hat{a} (X_{1})) (Y_{2} - \hat{b} (X_{2})) (A_{3} - \hat{a} (X_{3})) \bar{\phi}_{k} (X_{2})^{\top} \hat{\Sigma}^{-1} \bar{\phi}_{k} (X_{1}) \bar{\phi}_{k} (X_{1})^{\top} \hat{\Sigma}^{-1} \bar{\phi}_{k} (X_{3}) \right] \\
& - \left( \frac{2}{n (n - 1)} + \frac{4 (n - 2)}{n (n - 1)} \right) \left\{ \sfE \left[ (A - \hat{a} (X)) \bar{\phi}_{k} (X)^{\top} \right] \hat{\Sigma}^{-1} \sfE \left[ \bar{\phi}_{k} (X) (Y - \hat{b} (X)) \right] \right\}^{2} \\
= & \ \frac{1}{n (n - 1)} \Tr \left\{ \Sigma_{a, a} \hat{\Sigma}^{-1} \Sigma_{b, b} \hat{\Sigma}^{-1} + \left( \Sigma_{a, b} \hat{\Sigma}^{-1} \right)^{2} \right\} \\
& + \frac{n - 2}{n (n - 1)} \sfE [(Y - \hat{b} (X)) \bar{\phi}_{k} (X)^{\top}] \hat{\Sigma}^{-1} \Sigma_{a, a} \hat{\Sigma}^{-1} \sfE [\bar{\phi}_{k} (X) (Y - \hat{b} (X))] \\
& + \frac{n - 2}{n (n - 1)} \sfE [(A - \hat{a} (X)) \bar{\phi}_{k} (X)^{\top}] \hat{\Sigma}^{-1} \Sigma_{b, b} \hat{\Sigma}^{-1} \sfE [\bar{\phi}_{k} (X) (A - \hat{a} (X))] \\
& + \frac{2 (n - 2)}{n (n - 1)} \sfE [(A - \hat{a} (X)) \bar{\phi}_{k} (X)^{\top}] \hat{\Sigma}^{-1} \Sigma_{a, b} \hat{\Sigma}^{-1} \sfE [\bar{\phi}_{k} (X) (Y - \hat{b} (X))] \\
& - \frac{2 (2 n - 3)}{n (n - 1)} \left\{ \sfE \left[ (A - \hat{a} (X)) \bar{\phi}_{k} (X)^{\top} \right] \hat{\Sigma}^{-1} \sfE \left[ \bar{\phi}_{k} (X) (Y - \hat{b} (X)) \right] \right\}^{2} \\
= & \ \frac{1}{n (n - 1)} \Tr \left\{ \Sigma_{a, a} \hat{\Sigma}^{-1} \Sigma_{b, b} \hat{\Sigma}^{-1} + \left( \Sigma_{a, b} \hat{\Sigma}^{-1} \right)^{2} \right\} \\
& + \frac{n - 2}{n (n - 1)} \left( \alpha^{\top} \hat{\Sigma}^{-1} \Sigma_{b, b} \hat{\Sigma}^{-1} \alpha + \beta^{\top} \hat{\Sigma}^{-1} \Sigma_{a, a} \hat{\Sigma}^{-1} \beta + 2 \alpha^{\top} \hat{\Sigma}^{-1} \Sigma_{a, b} \hat{\Sigma}^{-1} \beta \right) \\
& - \frac{2 (2 n - 3)}{n (n - 1)} (\alpha^{\top} \hat{\Sigma}^{-1} \beta)^{2}.
\end{align*}
Finally, we compute the covariance term.
\begin{align*}
& \ 2 \cov \left[ \hat{\psi}_{1, n}, \hat{U}_{n, 2} (\bar{\phi}_{k}) \right] \\
= & \ \cov \left[ \frac{1}{n} \sum_{i = 1}^{n} (A_{i} - \hat{a} (X_{i})) (Y_{i} - \hat{b} (X_{i})), \frac{1}{n (n - 1)} \sum_{1 \leq i_{1} \neq i_{2} \leq n} (A_{i_{1}} - \hat{a} (X_{i_{1}})) \bar{\phi}_{k} (X_{i_{1}})^{\top} \hat{\Sigma}^{-1} \bar{\phi}_{k} (X_{i_{2}}) (Y_{i_{2}} - \hat{b} (X_{i_{2}})) \right] \\
= & \ \frac{2}{n^{2} (n - 1)} \sum_{i = 1}^{n} \sum_{1 \leq i_{1} \neq i_{2} \leq n} \cov \left[ (A_{i} - \hat{a} (X_{i})) (Y_{i} - \hat{b} (X_{i})), (A_{i_{1}} - \hat{a} (X_{i_{1}})) \bar{\phi}_{k} (X_{i_{1}})^{\top} \hat{\Sigma}^{-1} \bar{\phi}_{k} (X_{i_{2}}) (Y_{i_{2}} - \hat{b} (X_{i_{2}})) \right] \\
= & \ \frac{2}{n^{2} (n - 1)} \sum_{1 \leq i_{1} \neq i_{2} \leq n} \cov \left[ (A_{i_{1}} - \hat{a} (X_{i_{1}})) (Y_{i_{1}} - \hat{b} (X_{i_{1}})), (A_{i_{1}} - \hat{a} (X_{i_{1}})) \bar{\phi}_{k} (X_{i_{1}})^{\top} \hat{\Sigma}^{-1} \bar{\phi}_{k} (X_{i_{2}}) (Y_{i_{2}} - \hat{b} (X_{i_{2}})) \right] \\
& + \frac{2}{n^{2} (n - 1)} \sum_{1 \leq i_{1} \neq i_{2} \leq n} \cov \left[ (A_{i_{2}} - \hat{a} (X_{i_{2}})) (Y_{i_{2}} - \hat{b} (X_{i_{2}})), (A_{i_{1}} - \hat{a} (X_{i_{1}})) \bar{\phi}_{k} (X_{i_{1}})^{\top} \hat{\Sigma}^{-1} \bar{\phi}_{k} (X_{i_{2}}) (Y_{i_{2}} - \hat{b} (X_{i_{2}})) \right] \\
= & \ \frac{2}{n} \left\{ \begin{array}{c}
\sfE [(A - \hat{a} (X)) (Y - \hat{b} (X))^{2} \bar{\phi}_{k} (X)^{\top}] \hat{\Sigma}^{-1} \alpha + \sfE [(A - \hat{a} (X))^{2} (Y - \hat{b} (X)) \bar{\phi}_{k} (X)^{\top}] \hat{\Sigma}^{-1} \beta \\
- \, 2 \sfE [(A - \hat{a} (X)) (Y - \hat{b} (X))] \alpha^{\top} \hat{\Sigma}^{-1} \beta
\end{array} \right\}.
\end{align*}
\end{proof}

\end{document}